\documentclass[final,3p,11pt,times,authoryear]{elsarticle}

\usepackage{amsmath,amssymb,amsthm}
\usepackage{booktabs,array,tabularx}
\usepackage{graphicx}
\usepackage{placeins}
\usepackage{algorithm}
\usepackage{algpseudocode}
\usepackage{setspace}

\allowdisplaybreaks[3]

\theoremstyle{definition}
\newtheorem{definition}{Definition}
\theoremstyle{plain}
\newtheorem{theorem}{Theorem}
\newtheorem{proposition}{Proposition}
\newtheorem{lemma}{Lemma}

\theoremstyle{remark}

\begin{document}

\begin{frontmatter}

\title{High-Multiplicity Flexible Job Shops: From Exact Recurrent Fluid Attainment to Structure-Guided Finite-Horizon Scheduling}

\author[sds]{Wenjun Zheng}
\ead{wenjunzheng@link.cuhk.edu.cn}
\author[sds]{Wei Qu}
\ead{weiqu@link.cuhk.edu.cn}
\author[sse]{Weilin Cai}
\ead{weilincai@link.cuhk.edu.cn}
\author[sds,sribd]{Jianfeng Mao\corref{cor1}}
\ead{jfmao@cuhk.edu.cn}
\cortext[cor1]{Corresponding author}

\address[sds]{School of Data Science, The Chinese University of Hong Kong,
Shenzhen (CUHK-Shenzhen), Guangdong 518172, China}
\address[sse]{School of Science and Engineering, The Chinese University of
Hong Kong, Shenzhen (CUHK-Shenzhen), Guangdong 518172, China}
\address[sribd]{Shenzhen Research Institute of Big Data, Guangdong 518172,
China}

\begin{abstract}
High-multiplicity flexible job shops involve many copies of a small set of job
types that must be scheduled on alternative machines. Fluid relaxations
provide scalable workload lower bounds, but their fractional machine
allocations do not define feasible schedules for individual jobs. We show
that, after a suitable finite scaling, an optimal fluid allocation can be
realized exactly by a feasible repeating discrete schedule. The construction
scales the fluid allocation to integer operation counts, places the resulting
operations in nonoverlapping machine intervals, repeats this arrangement, and
links operations across repetitions into individual jobs without moving any
interval, thereby enforcing job precedence while preserving machine
feasibility. For growing finite instances with the same job-type composition,
even with a fixed number of extra jobs, the gap between the optimal makespan
and the fluid lower bound remains bounded by a constant; hence the relative
gap vanishes as the instance grows. Guided by this repeated structure, we
develop type-based cyclic template replay (TCTR), which searches job-type
templates using an optimization model whose size does not grow with the number
of job copies and replays the selected template on the full instance. On 676
multiplicity-expanded public flexible-job-shop instances, TCTR is feasible in
every case and achieves a 1.54\% mean gap to the fluid lower bound, compared
with 5.32\% for a job-indexed adaptive large-neighborhood search and 6.18\%
for a hybrid genetic algorithm.
\end{abstract}


\end{frontmatter}

\section{Introduction}
\label{sec:introduction}

The flexible job shop scheduling problem (FJSP) assigns each operation to an
eligible machine and sequences the assigned operations subject to job
precedence and machine capacity. It extends the classical job shop to
multipurpose machines \citep{brucker1991,dauzere1997,hurink1994} and remains
strongly NP-hard through the job-shop special case \citep{garey1976}; recent
reviews summarize the broad range of routing and sequencing models that have
grown around this structure \citep{dauzere2024,jiang2023}. In many production
systems, however, the number of distinct job types is small, while many copies
of each type must be processed. This repeated structure motivates
high-multiplicity scheduling, where job types are described once and demand is
encoded by their multiplicities rather than by an explicit list of jobs
\citep{boudoukh2001identical,bertsimas1999asymptotic}.

With many repeated jobs, it is natural to first characterize the system at an
aggregate workload level before constructing a detailed schedule for every job
copy. Fluid relaxations provide such a view and have long been used in
job-shop scheduling to derive machine-load lower bounds and guide large-volume
scheduling
\citep{bertsimas1999asymptotic,bertsimas2002fluid,
nazarathy2010fluid,masin2014highmultiplicity}. In a flexible job shop, the same
aggregate perspective can also capture machine flexibility. For each operation
position of a job type, its many copies can be distributed among eligible
machines; allowing these allocation quantities to be fractional yields a
compact way to balance machine workloads. Related fluid models have already
used such continuous allocation information to initialize, guide, or learn
scheduling decisions in multiplicity flexible job shops
\citep{ding2024hybridfluid,ding2024multiplicity,
ding2024multipolicy,ding2025deep,yang2025ddqn}. Our focus is different: rather
than using the fluid allocation only to guide subsequent scheduling decisions,
we ask whether an optimal fluid allocation can itself be realized
exactly by a feasible discrete execution. The relaxation studied here does not
model jobs as continuous flow through prescribed operation routes; instead, it
directly divides the required operations at each job-type position among their
eligible machines. After a finite scaling makes these machine allocations
integral, the remaining challenge is to assign the resulting operation
occurrences to individual jobs so that the operations of every job are
processed in the required order, without changing the machine loading
specified by the fluid allocation.

Fluid relaxation is one way to exploit the repetition inherent in
high-multiplicity scheduling by aggregating many job copies. A complementary
approach is to exploit the same repetition directly in the schedule itself:
when a small set of job types occurs many times, one may seek a short operating
pattern that can be repeated rather than construct the entire schedule copy by
copy. This idea underlies cyclic and recurrent scheduling, where repeated
execution is organized around a recurring operating pattern
\citep{roundy1992cyclic,hanen1994recurrent,levner2010cyclicsurvey}. Flexible
cyclic job-shop models further allow machine assignments to be chosen within
such repeated operation \citep{zhang2015cyclic,quinton2020flexible}. These
studies show how repeated schedules can be constructed and optimized once a
cyclic or recurrent regime is adopted, but they do not establish how the
performance of such schedules relates to the fluid workload. This
leaves a natural connection between the two perspectives above: can a feasible
recurrent schedule attain the fluid workload exactly? If so, what does such
recurrent attainment imply for finite high-multiplicity instances as the
number of job copies grows?

The proposed approach separates the machine-level schedule from the way
individual operations are grouped into jobs. Starting from an optimal fluid
allocation, the construction first scales the fractional machine allocations
until each machine receives an integer number of operations. These operations
are then scheduled on their assigned machines, and the resulting machine
schedule is repeated. At this point, machine capacity and the fluid workload
are already fixed, but the scheduled operations have not yet been grouped into
individual jobs. The construction is completed by grouping operations from
different repetitions so that every job follows its required operation order,
without changing any machine assignment or scheduled interval. This produces
a feasible repeating schedule with the same workload value as the optimal
fluid allocation.

This construction gives two consequences. First, the optimal fluid workload can be
attained exactly by a finite repeating discrete execution after a suitable
scaling. Second, the repeating part can be used as the main body of an ordinary
finite schedule, with only a bounded amount of work needed before and after it.
As a result, when the numbers of the different job types grow in fixed
proportions, the difference between the optimal finite makespan and the fluid
lower bound remains bounded by a constant that does not grow with the number
of job copies. The relative difference therefore converges to zero as the
instance grows.

The recurrent construction also reveals a useful scheduling structure for
finite instances. It shows that a high-multiplicity schedule need not be
designed independently for every job copy: the way job types share machine
capacity and the repeated order in which their operations are scheduled can be
represented by a much smaller job-type pattern, while job copies
are introduced when that pattern is realized. The exact pattern required by
the theorem, however, may involve more repetitions than are useful for a
particular finite instance, and exact return to the same system state is not
needed when the schedule has a fixed end.

This observation leads to type-based cyclic template replay (TCTR). TCTR keeps
the two parts of the recurrent construction that are useful for finite
scheduling---machine-load balancing by job type and repeated job-type
ordering---while allowing the pattern size and machine choices to adapt to the
actual finite instance. It searches for small integer patterns using the same
machine-load balancing structure as the fluid relaxation, tests how these patterns
perform when repeated over the complete set of jobs, and uses the selected
pattern to construct the final schedule. During this replay, concrete job
readiness is respected and machine assignments can be adjusted when this
improves the finite schedule. Thus the optimization remains at the job-type
level, with a model whose size does not increase with the number of job copies,
while the replay step produces the required operation-level schedule.

The computational study examines both parts of this development. A dense
fixed-composition study with instances of up to $10^5$ operations examines
whether finite schedules approach the fluid behavior predicted by the
recurrent theory. A broader study on 676 multiplicity-expanded public FJSP
instances evaluates TCTR as a finite-horizon scheduling method. TCTR produces
a feasible schedule for every instance and achieves a 1.54\% mean gap to the
fluid lower bound.

The main contributions of this paper are summarized as follows:
\begin{itemize}

\item
\textbf{Exact recurrent fluid attainment.}
The optimal fluid workload can be attained exactly by a feasible repeating
discrete schedule after a suitable finite scaling.

\item
\textbf{A finite-instance performance guarantee.}
Under fixed-composition growth, the optimal makespan remains within a constant
additive gap of the fluid lower bound, and the relative gap therefore vanishes
as the instance grows.

\item
\textbf{A theory-guided finite-horizon scheduling method.}
The recurrent structure is translated into type-based cyclic template replay
(TCTR), which searches compact job-type patterns rather than individual job
copies, with an optimization model whose size does not grow with the
multiplicities.

\item
\textbf{Computational validation at scale.}
The theoretical behavior is examined on instances with up to $10^5$
operations, while TCTR is evaluated on 676 multiplicity-expanded public FJSP
instances and achieves a 1.54\% mean gap to the fluid lower bound.

\end{itemize}

The remainder of the paper is organized as follows.
Section~\ref{sec:literature} reviews the related literature.
Section~\ref{sec:problem} formulates the high-multiplicity FJSP, its fluid
relaxation, and the exact event-driven representation used for recurrence.
Section~\ref{sec:mmc} establishes exact recurrent fluid attainment and the
fixed-composition additive bound. Sections~\ref{sec:tctr} and \ref{sec:experiments}
develop and evaluate TCTR, respectively.
Section~\ref{sec:conclusion} concludes.

\section{Literature Review}
\label{sec:literature}

The literature most closely related to this study develops along two
complementary ways of exploiting repetition in high-multiplicity scheduling.
One aggregates large numbers of repeated job copies through fluid descriptions
of workload; the other represents repetition explicitly through cyclic or
recurrent schedules. In flexible job shops, machine allocation becomes part of
the first view, while machine assignment can also be incorporated into the
second. The discussion
below follows this progression and focuses on the relation between fluid
workload and recurrent discrete execution. General FJSP routing and sequencing
are reviewed by \citet{dauzere2024} and \citet{jiang2023}.

\subsection{Fluid-to-finite scheduling under fixed machine routes}
\label{sec:lit_high_multiplicity}

High-multiplicity scheduling represents a fixed set of job types by their
multiplicities rather than by explicitly listing every job copy. Early work
exploited repeated jobs directly in job-shop scheduling
\citep{boudoukh2001identical}. In large-volume settings,
\citet{bertsimas1999asymptotic} developed asymptotically optimal scheduling
methods and established machine workload as a useful guide when many copies of
a fixed set of job types must be processed. Classical fluid relaxations
formalize this view by replacing discrete job copies with continuous amounts
of work.

An important question in this literature is how fluid information translates
back to finite schedules. For deterministic makespan minimization,
\citet{bertsimas2002fluid} developed a fluid synchronization algorithm and
proved a makespan bound equal to the fluid lower bound plus
$(I+2)P_{\max}J_{\max}$, where the additive term depends on the fixed job-type
structure rather than on the multiplicities. Related work studies fluid-based
large-volume scheduling under stochastic processing, broader job-type
structures, and prescribed initial and terminal inventories
\citep{dai2002fluid,nazarathy2010fluid,masin2014highmultiplicity}.
Together, these results establish that, under fixed machine routes, fluid
relaxations can support finite-schedule guarantees whose relative loss
decreases as the multiplicities grow.

In a flexible job shop, the fluid description must also account for machine
allocation. The machine used by an operation is no longer fixed, so repeated
operations at the same job-type position may be distributed among several
eligible machines. This changes the fluid problem itself: workload depends not
only on the multiplicities of the job types, but also on how their operations
are allocated across eligible machines. This leads to the multiplicity FJSP
literature discussed next.

\subsection{Fluid-based multiplicity FJSP with flexible machine assignment}
\label{sec:lit_mfjsp}

When many copies of the same job type are present, machine allocation can be
described by how the repeated operations at each operation position are
distributed among eligible machines. Allowing these machine-allocation
quantities to be continuous provides a natural fluid description of machine
flexibility and a way to balance machine workloads. Recent multiplicity FJSP
studies have used fluid information for initialization and online tracking,
adaptive search, reinforcement-learning state and action design, and
learning-assisted evolutionary optimization
\citep{ding2024hybridfluid,ding2024multiplicity,ding2024multipolicy,
ding2025deep,yang2025ddqn}.

These studies show that fluid information remains useful when machine
assignments are flexible. Their main concern is how a fluid allocation can guide
scheduling decisions for a finite instance. A different question concerns the
fluid allocation itself: whether the optimal fluid workload can also be attained
by a feasible discrete execution with the same long-run value. This question
becomes particularly natural when the repetition of job types is represented
directly in the schedule rather than only through the fluid relaxation.

\subsection{Recurrent and cyclic scheduling}
\label{sec:lit_cyclic}

Repeated job types can also be exploited by repeating the schedule itself.
Cyclic and recurrent scheduling develop this idea by seeking operating
patterns that can be repeated while maintaining feasibility across
repetitions. \citet{roundy1992cyclic} studied cyclic schedules for job shops
with identical jobs, while \citet{hanen1994recurrent} formulated the
recurrent job shop and analyzed its critical circuits. \citet{levner2010cyclicsurvey}
surveyed the complexity of cyclic scheduling problems, and
\citet{kimbrel2008highmultiplicity} studied high-multiplicity cyclic job shops
with identical jobs and fixed machine routes.

Machine flexibility has also been incorporated into cyclic scheduling.
\citet{zhang2015cyclic} considered a flexible job shop with time-window,
transport, and resource-capacity constraints, including products generated in
prescribed proportions. \citet{quinton2020flexible} jointly optimized machine
assignments and operation sequences in a flexible cyclic job shop. These
studies establish that repeated execution and flexible machine assignment can
be handled within the same scheduling framework.

The recurrent scheduling literature starts from a repeated discrete execution
and asks how that execution should be constructed or optimized.
Minimum-cycle-mean theory characterizes the value of cycles in finite directed
graphs \citep{karp1978cyclemean}, while max-plus and timed-event-graph theory
relate critical circuits to the long-run behavior of recurrent systems
\citep{baccelli1992synchronization}. These tools characterize performance once
a recurrent representation has been specified. Fluid scheduling starts from a
different object: a workload target obtained from a continuous relaxation.
The two perspectives therefore describe different ways of using the same
underlying repetition in job types.

\subsection{Connecting fluid workload and recurrent execution}
\label{sec:lit_synthesis}

Taken together, these literatures establish two complementary views of
high-multiplicity scheduling. Fluid methods characterize workload and, under
fixed machine routes, connect the fluid lower bound to finite-schedule
performance. Their extensions to multiplicity FJSP show how flexible machine
allocation can be represented in a fluid relaxation and used to guide
scheduling for finite instances. Recurrent scheduling, meanwhile, shows how a
repeated discrete schedule can be constructed and optimized, including
settings in which machine assignments remain flexible.

What remains unresolved is the relation between these two views in a
high-multiplicity flexible job shop: whether the optimal fluid workload can be
attained by a feasible recurrent discrete schedule. Establishing this relation
also raises a finite-instance question: whether recurrent attainment can yield
a performance guarantee for finite instances as their multiplicities grow.
The present study establishes this connection constructively and then uses the
repeated structure revealed by the recurrent result to organize
finite-horizon scheduling.

\FloatBarrier

\section{Problem Setting, Fluid Relaxation, and Event-Driven Representation}
\label{sec:problem}
\label{sec:model}

This section establishes the mathematical setting for the recurrent results
that follow. We first formulate the finite high-multiplicity FJSP and its
fixed-composition growth and fluid relaxation, and then introduce an exact
event-driven representation of finite schedules. The fluid relaxation supplies
the fluid workload, while the event-driven representation supplies the discrete state
needed to define recurrence.

\subsection{Multiplicity FJSP}

Let $\mathcal M=\{M_1,\ldots,M_m\}$ be the set of physical machines and
let $i=1,\ldots,R$ index the job types. Job type $i$ consists of an ordered
sequence of $H_i\geq1$ operation positions. At position $h$, the operation may be
processed on any machine in the nonempty eligible set
$\mathcal E_{i,h}\subseteq\mathcal M$ and requires deterministic processing
time $p_{i,h}(M)\in\mathbb N_+$ on $M\in\mathcal E_{i,h}$. Define the set of
type-position pairs
\begin{equation}
    \mathcal P:=\{(i,h):1\leq i\leq R,\ 1\leq h\leq H_i\}.
    \label{eq:type_position_set_ejor}
\end{equation}
The operation sequences, eligible sets, processing times, and physical
machines comprise the fixed shop structure and remain unchanged as job
multiplicities vary.

A multiplicity vector
$\boldsymbol\nu=(\nu_1,\ldots,\nu_R)^\top\in\mathbb N_0^R$ specifies
$\nu_i$ interchangeable copies of job type $i$. Copy $k$ is denoted by
$J_{i,k}$, $k=1,\ldots,\nu_i$, and its operation at position $h$ by
$O_{i,k,h}$. The instance is therefore described compactly by the fixed shop
structure and $\boldsymbol\nu$, whereas any realized schedule must assign and
sequence every operation occurrence.

All job copies are available at time zero, processing is nonpreemptive, and
there are no setup times, transfer delays, or blocking constraints.
A schedule $(u,\mathbf s)$ specifies a machine
$u_{i,k,h}\in\mathcal E_{i,h}$ and a start time
$s_{i,k,h}\in\mathbb R_{\geq0}$ for every $O_{i,k,h}$. It is feasible if
\begin{align}
    s_{i,k,h+1}
    &\geq s_{i,k,h}+p_{i,h}(u_{i,k,h}),
    &&h=1,\ldots,H_i-1,
    \label{eq:multiplicity_precedence_ejor}
\end{align}
for every job copy, and if the half-open processing intervals of any two
operations assigned to the same physical machine are disjoint. Its makespan is
\begin{equation}
    C_{\max}(u,\mathbf s)
    :=\max_{\substack{1\leq i\leq R,\ 1\leq k\leq\nu_i\\
                      1\leq h\leq H_i}}
      \{s_{i,k,h}+p_{i,h}(u_{i,k,h})\}.
    \label{eq:makespan_multiplicity_ejor}
\end{equation}
The optimal makespan for multiplicity vector $\boldsymbol\nu$ is
\begin{equation}
    C_{\max}^*(\boldsymbol\nu)
    :=\min\{C_{\max}(u,\mathbf s):(u,\mathbf s)
    \text{ is feasible for }\boldsymbol\nu\}.
    \label{eq:general_multiplicity_objective_ejor}
\end{equation}
For the empty instance, set $C_{\max}^*(\mathbf0):=0$.

\subsection{Fixed-composition growth and fluid relaxation}
\label{subsec:fixed_mix_instances_ejor}

Each $\boldsymbol\nu$ defines one finite multiplicity FJSP instance. To study
how such instances scale while the shop structure remains fixed, we
consider families generated by a nonzero growth vector
$\mathbf b\in\mathbb N_0^R\setminus\{\mathbf0\}$ and a fixed offset
$\mathbf r\in\mathbb N_0^R$:
\begin{equation}
    \boldsymbol\nu(q):=q\mathbf b+\mathbf r,
    \qquad q\in\mathbb N_0.
    \label{eq:fixed_mix_family_definition_ejor}
\end{equation}
One unit increase in $q$ adds $b_i$ copies of type $i$. Zero components of
$\mathbf b$ allow only a subset of job types to grow, while the offset
$\mathbf r$ records demand that remains fixed as $q$ increases. The
theoretical guarantees concern these fixed-composition families, but every
$\boldsymbol\nu(q)$ remains an ordinary finite FJSP instance.

\label{subsec:fluid_relaxation_ejor}

For the fixed shop structure and any demand vector
$\mathbf v\in\mathbb R_{\geq0}^R$, let $x_{i,h,M}$ denote the amount of
type-$i$ demand at position $h$ assigned fractionally to machine $M$. The
fluid relaxation minimizes the largest resulting machine workload:
\begin{equation}
\begin{aligned}
    \rho_F(\mathbf v):=\min_{\rho,x}\quad &\rho\\
    \text{subject to}\quad
    &\sum_{M\in\mathcal E_{i,h}}x_{i,h,M}=v_i,
    &&i=1,\ldots,R,\ h=1,\ldots,H_i,\\
    &\sum_{i=1}^R
      \sum_{\substack{1\leq h\leq H_i\\M\in\mathcal E_{i,h}}}
      p_{i,h}(M)x_{i,h,M}\leq\rho,
    &&M\in\mathcal M,\\
    &x_{i,h,M}\geq0.
\end{aligned}
\label{eq:multitype_fluid_lp_ejor}
\end{equation}
An optimal $x$ is an \emph{optimal fluid allocation}. The relaxation retains
the demand balance and eligible-machine restrictions at every type-position
pair, but it does not identify individual job copies or impose the temporal
linking of their successive operations.

The value $\rho_F$ is componentwise nondecreasing and positively homogeneous:
$\rho_F(\alpha\mathbf v)=\alpha\rho_F(\mathbf v)$ for every $\alpha\geq0$.
For later reporting, we use the equivalent notation
\begin{equation}
    \mathrm{LB}_F(\boldsymbol\nu):=\rho_F(\boldsymbol\nu)
    \label{eq:fluid_lower_bound_definition_ejor}
\end{equation}
for the fluid lower bound of a complete finite demand vector.
Indeed, the machine-assignment counts of any feasible schedule satisfy the
demand balances in \eqref{eq:multitype_fluid_lp_ejor}, and machine nonoverlap
implies that every resulting machine load is at most its makespan. Hence
\[
    \rho_F(\boldsymbol\nu)
    =\mathrm{LB}_F(\boldsymbol\nu)
    \leq C_{\max}^*(\boldsymbol\nu).
\]

The fluid relaxation therefore supplies both the fluid workload and
a lower bound for every complete finite demand vector. What it does not supply
is the discrete system state needed to define recurrence. We therefore turn to
an event-driven representation that removes interchangeable labels while
preserving the finite scheduling problem exactly.

\subsection{Event-driven state}

Finite multiplicities create interchangeable copies of each job type, and the
fixed shop structure may also contain interchangeable physical machines.
Two physical machines belong to the same \emph{machine type} if they have the same
eligibility status and, whenever eligible, the same processing time at every
$(i,h)\in\mathcal P$. Let $\mathcal M_1,\ldots,\mathcal M_L$ be the resulting
partition of $\mathcal M$, let $\widehat M_l$ be a fixed representative of
$\mathcal M_l$, and let
$\bar m_l:=|\mathcal M_l|$. Eligibility of $\widehat M_l$ and the value
$p_{i,h}(\widehat M_l)$ therefore depend only on machine type $l$. Write
\[
    \bar{\mathbf m}:=(\bar m_1,\ldots,\bar m_L)^\top .
\]

At decision stage $t\in\mathbb N_0$, let
$T_t\in\mathbb R_{\geq0}$ denote elapsed schedule time. Absolute time does not
affect admissibility and is therefore tracked separately from the state
\begin{equation}
    X_t=(\mathbf n_t,\mathbf m_t,\mathcal W_t).
    \label{eq:model_state}
\end{equation}
The stage index $t$ records the action sequence. Assignment actions leave
$T_t$ unchanged, whereas temporal-advancement actions increase it.

The waiting-job vector is
\begin{equation}
    \mathbf n_t=(n_{i,h,t})_{(i,h)\in\mathcal P}
    \in\mathbb N_0^{|\mathcal P|},
\end{equation}
where $n_{i,h,t}$ counts idle copies of job type $i$ whose next operation is
at position $h$. Thus $n_{i,1,t}$ counts unstarted jobs of type $i$, while
coordinates with $h\geq2$ count admitted jobs waiting for their next
operation. The idle-machine vector is
\begin{equation}
    \mathbf m_t=(m_{1,t},\ldots,m_{L,t})^\top,
    \qquad 0\leq m_{l,t}\leq\bar m_l,
\end{equation}
where $m_{l,t}$ counts idle machines in $\mathcal M_l$.

Active operations are represented by a finite multiset $\mathcal W_t$. A
descriptor $(l,i,h,\delta)$ records one machine in $\mathcal M_l$ processing
position $h$ of a type-$i$ job, with remaining processing time $\delta>0$.
Identical descriptors may occur, so braces and the operations $\uplus$ and
$\setminus$ involving $\mathcal W_t$ are interpreted in the multiset sense.
State consistency requires
\begin{equation}
    \sum_{(l',i,h,\delta)\in\mathcal W_t}\mathbb I(l'=l)
    =\bar m_l-m_{l,t},
    \qquad l=1,\ldots,L.
    \label{eq:machine_state_consistency}
\end{equation}

For multiplicity vector $\boldsymbol\nu$, the initial state is
\begin{equation}
    X_0(\boldsymbol\nu)
    =(\mathbf n_0,\bar{\mathbf m},\emptyset),
    \qquad
    n_{i,1,0}=\nu_i,
    \qquad
    n_{i,h,0}=0\quad(h\geq2),
    \qquad T_0=0,
    \label{eq:model_initial_state}
\end{equation}
and the terminal state is
$X_{\mathrm{term}}=(\mathbf0,\bar{\mathbf m},\emptyset)$.

\subsection{Event-driven actions and dynamics}

An assignment action $a_t=(l,i,h)$ starts position $h$ of one waiting type-$i$
job on one idle machine in $\mathcal M_l$. The admissible assignments at
$X_t$ are
\begin{equation}
    \mathcal A_{\mathrm{assgn}}(X_t)
    :=
    \left\{
    (l,i,h):1\leq l\leq L,\ (i,h)\in\mathcal P
    \;\middle|\;
    n_{i,h,t}>0,\ m_{l,t}>0,\
    \widehat M_l\in\mathcal E_{i,h}
    \right\}.
    \label{eq:admissible_assignment_ejor}
\end{equation}
If $a_t=(l,i,h)\in\mathcal A_{\mathrm{assgn}}(X_t)$, then
\begin{align}
    \mathbf n_{t+1}&=\mathbf n_t-\mathbf e_{i,h},\\
    \mathbf m_{t+1}&=\mathbf m_t-\mathbf e_l,\\
    \mathcal W_{t+1}
    &=\mathcal W_t\uplus
    \{(l,i,h,p_{i,h}(\widehat M_l))\},\\
    T_{t+1}&=T_t,
    \label{eq:assignment_transition_ejor}
\end{align}
where $\mathbf e_{i,h}$ and $\mathbf e_l$ are standard basis vectors of the
appropriate dimensions.

Because copies of the same job type and physical machines in the same machine
type are interchangeable, every assignment in this representation can be
lifted to a discrete execution. Select one matching waiting copy $J_{i,k}$ and one idle
physical machine $M\in\mathcal M_l$, and set
$u_{i,k,h}=M$ and $s_{i,k,h}=T_t$. The concrete execution retains this
job-machine association until completion, but the state representation does not need
to store either label.

The temporal-advancement operator $\mathcal F$ moves the clock to the next
completion event and is admissible if and only if
$\mathcal W_t\neq\emptyset$.
Define the time increment and completing submultiset by
\begin{equation}
    \Delta_t:=\min\{\delta\mid(l,i,h,\delta)\in\mathcal W_t\},
    \qquad
    \mathcal W_t^*
    :=\{(l,i,h,\delta)\in\mathcal W_t\mid\delta=\Delta_t\}.
\end{equation}
If $a_t=\mathcal F$, then
\begin{align}
    \mathcal W_{t+1}
    &=\{(l,i,h,\delta-\Delta_t)
      \mid(l,i,h,\delta)\in\mathcal W_t\setminus\mathcal W_t^*\},\\
    \mathbf m_{t+1}
    &=\mathbf m_t+
      \sum_{(l,i,h,\delta)\in\mathcal W_t^*}\mathbf e_l,\\
    \mathbf n_{t+1}
    &=\mathbf n_t+
      \sum_{\substack{(l,i,h,\delta)\in\mathcal W_t^*\\h<H_i}}
      \mathbf e_{i,h+1},\\
    T_{t+1}&=T_t+\Delta_t,
    \label{eq:advancement_transition_ejor}
\end{align}
where a descriptor with $h=H_i$ completes its job and contributes no new
waiting-job component. Simultaneous completions are represented with their
multiset multiplicities.

\subsection{Exactness for makespan}

We now show that restricting starts to the event epochs represented above is
without loss for makespan minimization. Only temporal advancement incurs a
positive stage cost. Define
\begin{equation}
    c(X_t,a_t)
    :=
    \begin{cases}
        0,&a_t\in\mathcal A_{\mathrm{assgn}}(X_t),\\
        \Delta_t,&a_t=\mathcal F.
    \end{cases}
    \label{eq:stage_cost_ejor}
\end{equation}
Thus $c(X_t,a_t)=T_{t+1}-T_t$. If a sequence of $K$ admissible actions
reaches $X_{\mathrm{term}}$, then
\begin{equation}
    \sum_{t=0}^{K-1}c(X_t,a_t)
    =T_K-T_0=C_{\max}.
    \label{eq:model_makespan_identity_ejor}
\end{equation}
It remains to verify that an optimal finite schedule can always be represented
by such an event-driven action sequence.

\begin{definition}[Completion-event schedule]
\label{def:completion_event_ejor}
A feasible schedule $(u,\mathbf s)$ is a \emph{completion-event schedule} if
every operation starts either at time zero or at the completion time of another
operation. Thus, for each $O_{i,k,h}$, either $s_{i,k,h}=0$ or
\[
    s_{i,k,h}
    =s_{i',k',h'}+p_{i',h'}(u_{i',k',h'})
\]
for some $i',k',h'$.
\end{definition}

\begin{proposition}[Model-generated schedules]
\label{prop:model_generated_ejor}
A feasible schedule for a multiplicity FJSP instance can be produced by the
event-driven representation if and only if it is a completion-event schedule.
\end{proposition}

\begin{proof}
Let $(u,\bar{\mathbf s})$ be a completion-event schedule and list its distinct
positive completion epochs as
$0<\theta_1<\cdots<\theta_E$, with $\theta_0:=0$. At each epoch $\theta_e$,
serialize the operations prescribed to start at that time. Every such
assignment is admissible because schedule feasibility makes its job ready and
its assigned machine eligible and idle; operations starting simultaneously use
distinct physical machines. After these assignments, the event-driven state
and $(u,\bar{\mathbf s})$ contain the same active operations with the same
remaining processing times. For $e<E$, the operator $\mathcal F$ therefore
advances exactly to $\theta_{e+1}$ and realizes the same simultaneous
completions. Induction over the completion epochs reproduces the schedule.

Conversely, every assignment generated by the representation starts at the
current clock value. Assignment actions leave the clock unchanged, whereas
every positive temporal advance terminates at an operation-completion epoch.
Since $T_0=0$, every generated operation start is either zero or a completion
time.
\end{proof}

\begin{lemma}[Dominance of completion-event schedules]
\label{lem:completion_event_dominance_ejor}
For every feasible schedule $(u,\mathbf s)$, there is a completion-event
schedule $(u,\bar{\mathbf s})$ that preserves the operation order on every
machine and satisfies
\begin{equation}
    C_{\max}(u,\bar{\mathbf s})\leq C_{\max}(u,\mathbf s).
    \label{eq:completion_dominance_ejor}
\end{equation}
\end{lemma}

\begin{proof}
Fix the machine assignments $u$ and the machine orders induced by
$\mathbf s$. Scan the operations by nondecreasing original start time and
start each operation as early as permitted by time zero and the shifted
completion times of its immediate job and machine predecessors, when present.
Because processing times are positive, both predecessors appear earlier in the
scan. Induction therefore gives
$\bar s_{i,k,h}\leq s_{i,k,h}$ for every operation. Feasibility and all machine
orders are preserved, and the makespan cannot increase. Every positive shifted
start is equal to the completion time of one of its predecessors, so
$(u,\bar{\mathbf s})$ is a completion-event schedule.
\end{proof}

Proposition~\ref{prop:model_generated_ejor} and
Lemma~\ref{lem:completion_event_dominance_ejor} show that some optimal schedule
is representable by the event-driven representation. Together with
\eqref{eq:model_makespan_identity_ejor}, this proves that minimizing path cost
from $X_0(\boldsymbol\nu)$ to $X_{\mathrm{term}}$ yields
$C_{\max}^*(\boldsymbol\nu)$. The representation is therefore not a relaxation: it
removes interchangeable labels and unnecessary continuous-time start epochs
without changing the finite-instance optimum. Its dominant completion-event
subclass is consistent with classical schedule-generation results
\citep{artigues2005schedule,giffler1960algorithms,sprecher1995semiactive}.

The fluid relaxation and this exact state representation provide the two
ingredients needed for the recurrent construction. Section~\ref{sec:mmc}
externalizes the unstarted supply coordinates while retaining all admitted
unfinished jobs and machine states, and asks whether the fluid workload
$\rho_F(\mathbf b)$ can be attained by a finite state-returning execution.

\section{Exact Recurrent Fluid Attainment and Finite-Instance Guarantees}
\label{sec:mmc}

This section establishes exact recurrent attainment of the fluid workload and
its finite-instance consequence. We first define recurrent state return and
finite boundary connectors, then use period-shifted slot linking to convert a
rational fluid allocation into a recurrent discrete execution. Applying the
construction to an optimal fluid allocation gives exact attainment, and
repeating the attaining period between bounded boundary pieces yields the
fixed-composition additive guarantee. Example~S.1 in the supplement illustrates the
single-job-type construction.

\subsection{Recurrent state and finite boundary connectors}
\label{subsec:fixed_mix_growth}

Fix a nonzero growth vector
$\mathbf b\in\mathbb N_0^R\setminus\{\mathbf0\}$. The coordinates $n_{i,1}$
record only unstarted supply, not work already admitted to the system. For
recurrent production, replace these coordinates by unlimited external supplies
and retain every coordinate for positions $h\geq2$:
\begin{equation}
    \operatorname{proj}(X)
    :=\bigl((n_{i,h})_{(i,h)\in\mathcal P,\ h\geq2},
            \mathbf m,\mathcal W\bigr).
    \label{eq:type_position_projection_ejor}
\end{equation}
The projected state therefore retains every admitted unfinished job and the
complete machine state.

A \emph{closed walk} is a finite walk returning to its initial projected state
and may repeat vertices; \emph{cycle} is used in this sense. A periodic
realization additionally repeats every operation interval after a fixed time
shift.

\begin{definition}[Type-position projected graph and balanced closed walks]
\label{def:fixed_mix_joint_cycles_ejor}
The weighted graph $\mathcal G_\infty^{(R)}$ has as ordinary vertices the
projected states reachable from
$v_0=(\mathbf0,\bar{\mathbf m},\emptyset)$ under unlimited external
position-$(i,1)$ supply. An admissible assignment gives a zero-weight edge,
and an assignment at position $(i,1)$ admits one new job of type $i$. If an
$\mathcal F$ action of duration $\Delta$ completes several descriptors, forced
auxiliary edges
serialize those completions in a fixed order with total weight $\Delta$.
Contracting each forced chain recovers the corresponding transition of the
event-driven representation.

For a closed walk $C$, let $w(C)$ be its total weight, $|C|$ its edge count,
and
$\mathbf A(C)=(A_1(C),\ldots,A_R(C))^\top$ its job-type admission vector.
The walk is $\mathbf b$-\emph{balanced} if
$\mathbf A(C)=y_C\mathbf b$ for some $y_C\in\mathbb N_+$.
Its admission-normalized value is $w(C)/y_C$.
\end{definition}

Unfinished-job inventories can make the projected graph infinite, so the
analysis works directly with finite walks. Every finite walk lifts to a
feasible event-driven execution by supplying a distinct job copy at each
admission edge. Conversely, externalizing the unstarted coordinates of any
finite execution gives a walk in the projected graph. Both operations preserve
actions, elapsed time, and machine feasibility.

Flow conservation applies separately to every operation position. A
$\mathbf b$-balanced closed walk therefore contains $A_i(C)$ assignments and
completions at every position of job type $i$. Hence
\begin{equation}
    |C|
    =2\sum_{i=1}^R H_iA_i(C)
    =2y_C\sum_{i=1}^R b_iH_i.
    \label{eq:multitype_cycle_accounting_ejor}
\end{equation}
Thus the natural normalization of a recurrent walk is by the multiplier
$y_C$ in its admission vector, rather than by its edge count.

The recurrent boundary need not be empty, so finite filling and draining paths
are needed to embed a recurrent walk into an ordinary finite schedule. For
each job type and demand vector
$\mathbf v\in\mathbb R_{\geq0}^R$, define the serial processing bound
\begin{equation}
    P_i
    :=\sum_{h=1}^{H_i}
       \min_{M\in\mathcal E_{i,h}}p_{i,h}(M),
    \qquad
    P(\mathbf v):=\sum_{i=1}^R v_iP_i.
    \label{eq:multitype_serial_time_ejor}
\end{equation}

\begin{lemma}[Finite type-position connectors]
\label{lem:type_position_connectors_ejor}
Every ordinary vertex $v$ has a finite filling path from $v_0$ to $v$ and a
finite no-admission draining path from $v$ to $v_0$. Their weights depend on
$v$ and the fixed shop structure, not on any later number of cycle
repetitions. Moreover, every finite all-idle projected state is an ordinary
vertex. If it contains $\beta_i$ unfinished jobs of type $i$, its connectors
can be chosen serially so that
\begin{equation}
    T_v^{\mathrm{in}}+T_v^{\mathrm{out}}
    =P(\boldsymbol\beta),
    \qquad
    \boldsymbol\beta=(\beta_1,\ldots,\beta_R)^\top.
    \label{eq:multitype_connector_identity_ejor}
\end{equation}
\end{lemma}

\begin{proof}
The filling path to an ordinary vertex exists by reachability. To drain, first
complete every active descriptor and then process the finitely many waiting
jobs through their remaining positions without further admissions. Any finite
all-idle projected state can be reached by admitting its unfinished jobs and
processing their required operation-sequence prefixes serially. If each
operation uses a fastest eligible machine, the filling prefix and draining
suffix together process every operation of each unfinished job exactly once.
Summing by job type gives \eqref{eq:multitype_connector_identity_ejor}.
\end{proof}

\subsection{Period-shifted slot linking: from fluid allocation to recurrent execution}
\label{subsec:slot_linking_mechanism}

A fluid allocation fixes machine workloads but not the ownership of individual
operation occurrences. Period-shifted slot linking therefore freezes the
machine intervals first and restores job precedence only by relabelling slot
ownership across period copies; no machine interval is moved.

For a rational allocation $x$ satisfying the type-position balances for
$\mathbf b$, let its maximum machine load be
\begin{equation}
    \rho(x):=\max_{M\in\mathcal M}
    \sum_{i=1}^R
    \sum_{\substack{1\leq h\leq H_i\\M\in\mathcal E_{i,h}}}
       p_{i,h}(M)x_{i,h,M}
    \label{eq:multiroute_allocation_load_ejor}
\end{equation}

The construction has three steps. First, scale $x$ by a common denominator
$y$, pack the resulting integer occurrences into $[0,T)$ on each machine, and
repeat these type-position-labelled slot patterns every
$T=y\rho(x)$ time units. Second, for each job type $i$, label the $yb_i$ slots
at every position by lanes $k=1,\ldots,yb_i$. If $M_{i,h,k}$ and
$\varphi_{i,h,k}\in[0,T)$ are the machine and phase of one labelled slot, set
$z_{i,1,k}=0$ and recursively choose
\begin{equation}
z_{i,h+1,k}:=\min\left\{
d\in\mathbb Z_{\geq z_{i,h,k}}:
\varphi_{i,h+1,k}+dT
\geq
\varphi_{i,h,k}+z_{i,h,k}T+p_{i,h}(M_{i,h,k})
\right\}.
\label{eq:period_offset_recursion_ejor}
\end{equation}
Generation $g$ in lane $(i,k)$ uses period copy $g+z_{i,h,k}$ of its
position-$h$ slot. The map $g\mapsto g+z_{i,h,k}$ is a bijection on
$\mathbb Z$, so linking changes job ownership across period copies but never
moves a machine interval. Third, translating $g$ by one between consecutive
boundaries establishes state return.

Figure~\ref{fig:periodic_mmc_construction_ejor} illustrates the three steps:
fluid scaling fixes the periodic machine slots, period shifts link these slots
into precedence-feasible jobs without moving them, and consecutive period
boundaries recover the same projected state.
\begin{figure}[H]
    \centering
    \includegraphics[width=\textwidth,trim=4bp 3bp 2bp 4bp,clip]
    {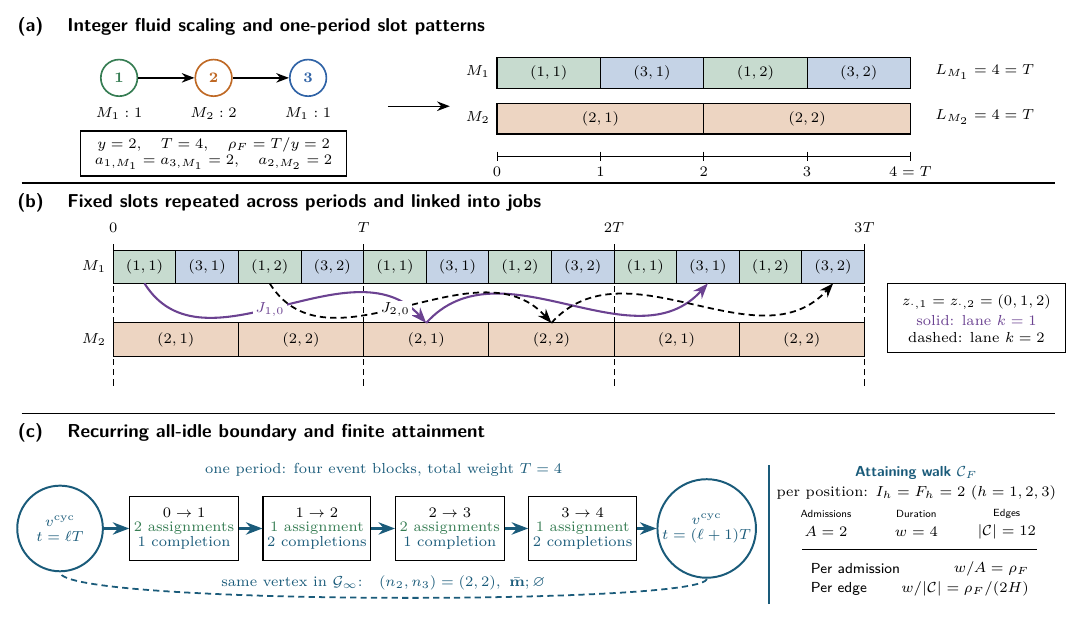}
    \setlength{\abovecaptionskip}{4pt}
    \caption{The three steps of period-shifted slot linking for one job type
    with three operation positions, $y=2$, and $T=4$: integer fluid scaling,
    cross-period job linking with fixed machine slots, and return to the same
    all-idle projected state.}
    \label{fig:periodic_mmc_construction_ejor}
\end{figure}

\begin{lemma}[Recurrent realization of a rational fluid allocation]
\label{lem:multiroute_slot_linking_ejor}
Fix nonzero $\mathbf b$ and let $x$ be any rational allocation satisfying the
type-position balances of \eqref{eq:multitype_fluid_lp_ejor} for
$\mathbf v=\mathbf b$. There exist $y\in\mathbb N_+$ and $T=y\rho(x)$ such
that a finite $\mathbf b$-balanced closed walk $C_x$ satisfies
\begin{equation}
    \mathbf A(C_x)=y\mathbf b,
    \qquad
    w(C_x)=T,
    \qquad
    |C_x|=2y\sum_{i=1}^R b_iH_i.
    \label{eq:multiroute_slot_linking_statistics_ejor}
\end{equation}
Its lifted execution is $T$-periodic and returns to the same finite all-idle
projected state.
\end{lemma}

\begin{proof}
Choose a common denominator $y$ such that
$a_{i,h,M}:=yx_{i,h,M}$ is integral for every type-position pair and eligible
machine. Every scaled machine load is then integral, so their maximum
$T=y\rho(x)$ is integral. Pack the occurrences assigned to each machine
consecutively in $[0,T)$ and repeat the resulting slot pattern every $T$ time
units. For job type $i$ with $b_i>0$, label its $yb_i$ slots at every position
by lanes
$k=1,\ldots,yb_i$; types with $b_i=0$ have no slots.

Apply \eqref{eq:period_offset_recursion_ejor} to every lane. Job
$J_{i,k,g}$ uses period copy $g+z_{i,h,k}$ of slot $(i,h,k)$. For each fixed
slot, $g\mapsto g+z_{i,h,k}$ is a bijection on $\mathbb Z$, and the recursion
enforces job precedence. Because all type-position-labelled slots were packed
before relabelling, linking cannot create a machine conflict. Every
nonzero-phase slot starts when the preceding slot on its machine completes. At
each period boundary, a maximum-load machine completes its last slot, so
phase-zero starts also occur at a completion epoch. The linked schedule is
therefore a completion-event schedule.

It remains only to verify finite state return. Let $c_{i,h,k,g}$ be the
completion time of position $h$ of $J_{i,k,g}$. At boundary $nT$, after
boundary completions and before phase-zero assignments, let $\eta$ record the
next position of each admitted generation, using $H_i+1$ when the job is
complete, and let $\beta$ count the resulting waiting jobs:
\begin{equation}
\begin{aligned}
    \eta_{i,k,g}(n)
    &:=\min\bigl(\{h:c_{i,h,k,g}>nT\}\cup\{H_i+1\}\bigr),
      &&g\leq n-1,\\
    \beta_{i,h}(n)
    &:=\sum_{k=1}^{yb_i}\sum_{g\leq n-1}
      \mathbf{1}\{\eta_{i,k,g}(n)=h\},
      &&h=2,\ldots,H_i.
\end{aligned}
    \label{eq:multiroute_boundary_inventory_ejor}
\end{equation}
Each slot finishes within its period, so the recursion increases
$z_{i,h,k}$ by at most one at each position and $z_{i,h,k}\leq h-1$. Hence
only generations $g>n-H_i$ can contribute at boundary $nT$. Moreover,
$c_{i,h,k,g+1}=c_{i,h,k,g}+T$, so
$\eta_{i,k,g+1}(n+1)=\eta_{i,k,g}(n)$ and every $\beta_{i,h}(n)$ repeats.
All machines are idle and the active set is empty at these boundaries; the
entire projected state is therefore finite and repeats. By
Lemma~\ref{lem:type_position_connectors_ejor}, this boundary state is
reachable and drainable. One period induces a finite closed walk with
admissions $y\mathbf b$ and weight $T$; its edge count follows from
\eqref{eq:multitype_cycle_accounting_ejor}.
\end{proof}

The bi-infinite generation index describes the steady pattern compactly; it
does not create an infinite boundary state. Each boundary has finite
unfinished inventory, so any finite number of periods can be placed between
finite connectors.

\subsection{Exact recurrent attainment of the fluid optimum}
\label{subsec:exact_fluid_attainment}

The preceding construction applies to any rational fluid allocation. Since the
fluid LP has rational data, it has a rational optimum, which gives the main
attainment result.

\begin{theorem}[Exact recurrent fluid attainment]
\label{thm:multitype_mmc_attainability_ejor}
For every nonzero $\mathbf b\in\mathbb N_0^R$, there exist
$y\in\mathbb N_+$ and a feasible $y\rho_F(\mathbf b)$-periodic
completion-event schedule that admits
$yb_i$ jobs of every type $i$ per period and returns to the same finite
all-idle projected state. One period induces a finite
$\mathbf b$-balanced closed walk $\mathcal C_{\mathbf b}$ satisfying
\begin{equation}
    \mathbf A(\mathcal C_{\mathbf b})=y\mathbf b,
    \qquad
    w(\mathcal C_{\mathbf b})=y\rho_F(\mathbf b),
    \qquad
    |\mathcal C_{\mathbf b}|
      =2y\sum_{i=1}^R b_iH_i.
    \label{eq:multitype_cycle_statistics_ejor}
\end{equation}
\end{theorem}

\begin{proof}
The fluid LP has a rational optimal allocation $x^*$. Since
$\rho(x^*)=\rho_F(\mathbf b)$,
Lemma~\ref{lem:multiroute_slot_linking_ejor} gives the stated periodic
execution and closed walk. The edge count follows from
\eqref{eq:multitype_cycle_accounting_ejor}.
\end{proof}

\paragraph{Admission-normalized closed-walk consequence.}
For every $k\in\mathbb N_+$, any feasible schedule containing at least $kb_i$
jobs of every type $i$ satisfies
\begin{equation}
    C_{\max}\geq k\rho_F(\mathbf b).
    \label{eq:multitype_fluid_lower_bound_ejor}
\end{equation}
Indeed, selecting $kb_i$ jobs of each type, counting their machine
assignments, and dividing by $k$ gives a feasible fluid allocation for
$\mathbf b$ whose maximum machine load is at most $C_{\max}/k$.

Now let $C$ be any $\mathbf b$-balanced closed walk with
$\mathbf A(C)=y_C\mathbf b$.
Lemma~\ref{lem:type_position_connectors_ejor} places $C^N$ between fixed
filling and draining paths, giving a finite schedule of makespan
$Nw(C)+K_C$, where $K_C$ is independent of $N$. Applying
\eqref{eq:multitype_fluid_lower_bound_ejor} with $k=Ny_C$ and letting
$N\to\infty$ yields $w(C)/y_C\geq\rho_F(\mathbf b)$. Conversely,
Theorem~\ref{thm:multitype_mmc_attainability_ejor} constructs
$\mathcal C_{\mathbf b}$ with normalized weight exactly
$\rho_F(\mathbf b)$. Hence
\begin{equation}
    \inf_{\substack{C\text{ closed and }\mathbf b\text{-balanced}}}
      \frac{w(C)}{y_C}
    =\rho_F(\mathbf b),
    \label{eq:multitype_fluid_equals_mmc_ejor}
\end{equation}
and the infimum is attained by $\mathcal C_{\mathbf b}$.

\subsection{From recurrent attainment to finite fixed-composition schedules}
\label{subsec:finite_fixed_mix_guarantee}

For a finite demand $q\mathbf b+\mathbf r$, the attaining period need not fit
an integer number of times. We therefore repeat the recurrent core as many
times as possible and absorb the fixed offset, the remainder modulo $y$, and
the entry and exit connectors into bounded boundary pieces. Only the number of
recurrent periods grows with $q$.

Fix the cycle $\mathcal C_{\mathbf b}$ from
Theorem~\ref{thm:multitype_mmc_attainability_ejor}, with multiplier $y$.
At one all-idle boundary, let
$\boldsymbol\beta=(\beta_1,\ldots,\beta_R)^\top$ count unfinished jobs by
job type, and set
\begin{equation}
    \sigma_{\boldsymbol\beta}
    :=\max_{i:b_i>0}\left\lceil\frac{\beta_i}{b_i}\right\rceil.
    \label{eq:multitype_boundary_scale_ejor}
\end{equation}
For any fixed $\mathbf r\in\mathbb N_0^R$, define the $q$-independent constant
\begin{equation}
    \Gamma_{\mathbf b,\mathbf r}
    :=P(\mathbf r)
      +(\sigma_{\boldsymbol\beta}+y-1)
       \bigl(P(\mathbf b)-\rho_F(\mathbf b)\bigr).
    \label{eq:multitype_additive_constant_ejor}
\end{equation}

\begin{theorem}[Uniform additive bound under fixed-composition growth]
\label{thm:multitype_additive_bound_ejor}
For every $q\in\mathbb N_0$, there is a feasible completion-event schedule
whose makespan $\widehat C_{\max}(q;\mathbf b,\mathbf r)$ satisfies
\begin{equation}
\begin{aligned}
    q\rho_F(\mathbf b)
    &\leq \rho_F(\boldsymbol\nu(q))
     =\mathrm{LB}_F(\boldsymbol\nu(q))\\
    &\leq C_{\max}^*(\boldsymbol\nu(q))
     \leq\widehat C_{\max}(q;\mathbf b,\mathbf r)\\
    &\leq q\rho_F(\mathbf b)+\Gamma_{\mathbf b,\mathbf r}.
\end{aligned}
\label{eq:multitype_additive_bound_ejor}
\end{equation}
Here $\mathrm{LB}_F(\boldsymbol\nu(q))$ is evaluated for the complete finite
demand vector $q\mathbf b+\mathbf r$ and therefore includes every offset copy
in $\mathbf r$, including job types with $b_i=0$. Consequently,
\begin{equation}
\begin{aligned}
    0\leq
    \widehat C_{\max}(q;\mathbf b,\mathbf r)
      -C_{\max}^*(\boldsymbol\nu(q))\\
    &\leq
    \widehat C_{\max}(q;\mathbf b,\mathbf r)
      -\mathrm{LB}_F(\boldsymbol\nu(q))
    \leq\Gamma_{\mathbf b,\mathbf r},
    \qquad
    \lim_{q\to\infty}
    \frac{\widehat C_{\max}(q;\mathbf b,\mathbf r)}
         {C_{\max}^*(\boldsymbol\nu(q))}
    =1.
\end{aligned}
    \label{eq:multitype_asymptotic_ratio_ejor}
\end{equation}
\end{theorem}

\begin{proof}
\emph{Fluid lower bounds.}
Positive homogeneity, demand monotonicity, and
\eqref{eq:fluid_lower_bound_definition_ejor} give
\[
    q\rho_F(\mathbf b)=\rho_F(q\mathbf b)
    \leq\rho_F(q\mathbf b+\mathbf r)
    \leq C_{\max}^*(\boldsymbol\nu(q)).
\]

\emph{Fixed boundary pieces.}
At the selected boundary, a serial filling connector creates each unfinished
job by processing its operation-sequence prefix, and a serial draining
connector completes its suffix. Equation~
\eqref{eq:multitype_connector_identity_ejor} gives
\[
    T_{\mathrm{in}}+T_{\mathrm{out}}
    =P(\boldsymbol\beta).
\]
The periodic construction has no boundary job with $b_i=0$, and the definition
of $\sigma_{\boldsymbol\beta}$ gives
$\sigma_{\boldsymbol\beta}\mathbf b-\boldsymbol\beta\geq\mathbf0$.

\emph{Stitching.}
For $q\geq\sigma_{\boldsymbol\beta}$, write
\[
    q-\sigma_{\boldsymbol\beta}=Ny+\ell,
    \qquad
    N\in\mathbb N_0,\quad 0\leq\ell<y.
\]
Process
\[
    \mathbf r
    +(\sigma_{\boldsymbol\beta}\mathbf b-\boldsymbol\beta)
    +\ell\mathbf b
\]
jobs serially, execute the filling connector, repeat
$\mathcal C_{\mathbf b}$ $N$ times, and drain. The admission vectors sum to
$\mathbf r+q\mathbf b=\boldsymbol\nu(q)$.

Figure~\ref{fig:finite_demand_mmc_stitching_ejor} summarizes this stitching:
only the recurrent core is repeated with $q$, whereas the serial remainder and
the filling and draining connectors remain bounded.
\begin{figure}[!t]
    \centering
    \includegraphics[width=\textwidth,trim=5bp 2bp 2bp 2bp,clip]
    {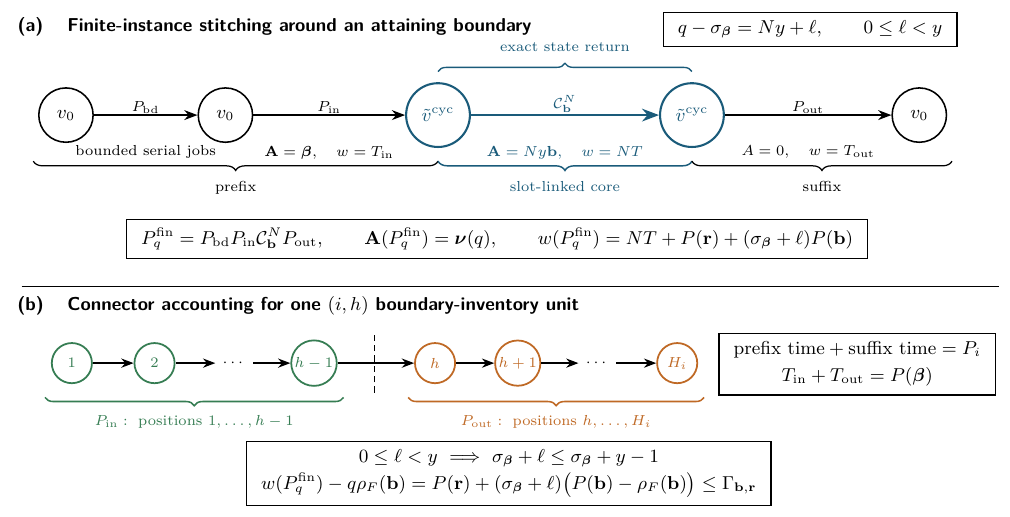}
    \setlength{\abovecaptionskip}{4pt}
    \caption{Finite-instance stitching of bounded boundary pieces around repeated
$\mathbf b$-balanced attaining cycles.}
    \label{fig:finite_demand_mmc_stitching_ejor}
\end{figure}

Using $w(\mathcal C_{\mathbf b})=y\rho_F(\mathbf b)$,
\begin{align*}
    \widehat C_{\max}(q;\mathbf b,\mathbf r)
    &=P(\mathbf r)
      +P(\sigma_{\boldsymbol\beta}\mathbf b-\boldsymbol\beta)
      +\ell P(\mathbf b)+T_{\mathrm{in}}
      +Ny\rho_F(\mathbf b)+T_{\mathrm{out}}\\
    &=q\rho_F(\mathbf b)+P(\mathbf r)
      +(\sigma_{\boldsymbol\beta}+\ell)
       \bigl(P(\mathbf b)-\rho_F(\mathbf b)\bigr)\\
    &\leq q\rho_F(\mathbf b)+\Gamma_{\mathbf b,\mathbf r}.
\end{align*}
Here $\ell\leq y-1$ and
$P(\mathbf b)\geq\rho_F(\mathbf b)$ because the serial allocation is
fluid-feasible. If $q<\sigma_{\boldsymbol\beta}$, serially processing all jobs
costs $P(\mathbf r)+qP(\mathbf b)$ and obeys the same bound because
$q\leq\sigma_{\boldsymbol\beta}-1$. Feasibility gives
\[
    C_{\max}^*(\boldsymbol\nu(q))
    \leq\widehat C_{\max}(q;\mathbf b,\mathbf r).
\]
The additive conclusions follow from the full chain, and the ratio converges
to one because $\rho_F(\mathbf b)>0$.
\end{proof}

The result permits zero components of $\mathbf b$, so only selected job types
need grow. It extends fixed-route fluid-to-finite guarantees
\citep{bertsimas2002fluid} to flexible machine assignment. The attaining
recurrent schedule is a structural certificate rather than a claim of
finite-instance optimality: for a particular finite target, its
denominator-clearing batch may be unnecessarily large. This motivates the
finite-horizon replay method in Section~\ref{sec:tctr}.
\section{Finite-Horizon Scheduling with Type-Based Cyclic Template Replay}
\label{sec:tctr}

The exact recurrent construction of Section~\ref{sec:mmc} identifies a
repeated job-type structure that attains the fluid workload, but its
denominator-clearing batch may be too coarse for a particular finite target.
Type-based cyclic template replay (TCTR) retains machine-load balancing by job
type and repeated job-type ordering while relaxing exact period
and projected-state return. It therefore searches smaller representative
patterns directly on the finite demand and constructs the complete schedule by
replay.

Figure~\ref{fig:bridge} summarizes the transition from the exact recurrent
construction of Section~\ref{sec:mmc} to the finite-horizon scheduling approach
developed below.
\begin{figure}[!ht]
    \centering
    \includegraphics[width=\textwidth]{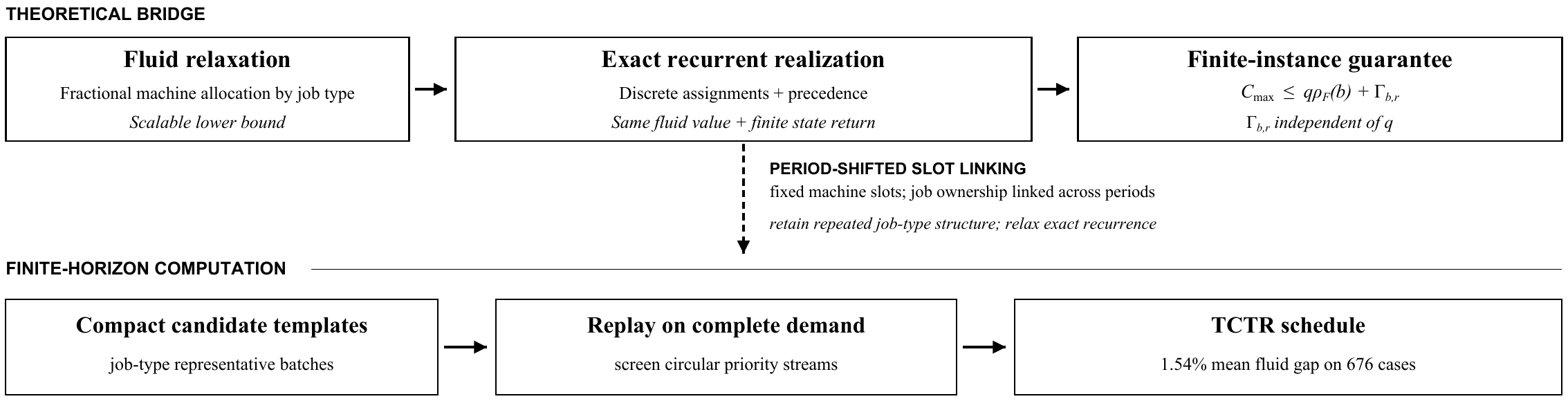}
    \setlength{\abovecaptionskip}{4pt}
    \caption{From exact recurrent fluid attainment to finite-horizon
    scheduling.}
    \label{fig:bridge}
\end{figure}
\FloatBarrier

\subsection{From recurrent structure to finite replay}
\label{subsec:tctr_exact_relaxation}

For target multiplicities $\boldsymbol\nu$, a candidate batch
$\mathbf z\in\mathbb N_0^R$ satisfies $z_i>0$ exactly when $\nu_i>0$ and
contains $z_i$ occurrences of every operation position of job type $i$.
These counts support repeated replay by job type but, unlike the cycles of
Section~\ref{sec:mmc}, do not imply projected-state return.
We therefore use \emph{circular priority stream} for the repeated ordering and
reserve \emph{cycle} for a schedule with verified state return.

TCTR proceeds in four steps. First, it generates a finite nonempty set
$\mathcal Z(\boldsymbol\nu)$ of candidate batches satisfying this condition. Second, a compact
integer model assigns the type-position occurrences of each batch to eligible
machines. Third, the resulting machine-labeled multiset is arranged into a
finite set of circular priority streams and screened by replay on the target instance.
Finally, the best screened type-position order is replayed once with
earliest-completion-time (ECT) machine selection, and the better schedule is
returned. Both candidate generation and replay use $\boldsymbol\nu$ directly;
no decomposition into a growth vector and fixed offset is required.

Figure~\ref{fig:tctr_template_replay_ejor} illustrates the key distinction:
stream wrap-around repeats the priority ordering but does not impose
projected-state return.

\begin{figure}[H]
    \centering
    \includegraphics[width=\textwidth,trim=4bp 5bp 5bp 4bp,clip]{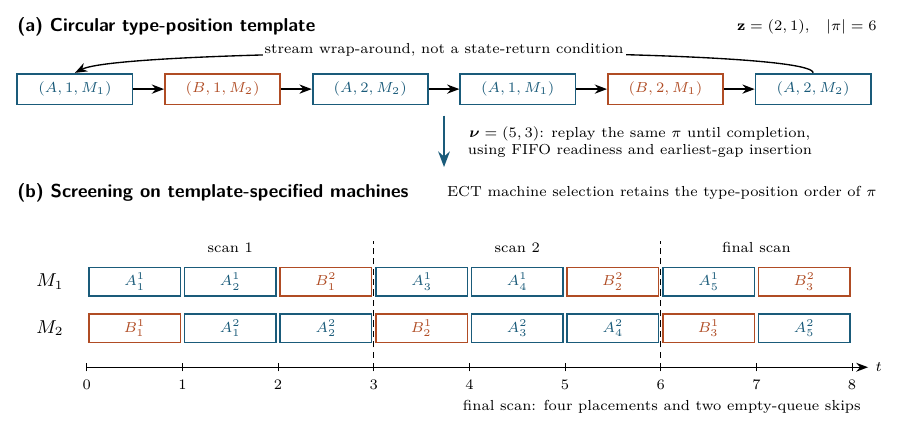}
    \setlength{\abovecaptionskip}{4pt}
    \caption{TCTR replay for $\boldsymbol\nu=(5,3)$ and $\mathbf z=(2,1)$.
    Screening on the template-specified machines preserves those machine
    assignments; ECT machine selection preserves the type-position order but
    may choose different machines.}
    \label{fig:tctr_template_replay_ejor}
\end{figure}

\subsection{Candidate templates and circular streams}
\label{subsec:tctr_replay}
\label{subsec:tctr_allocation}

For each $\mathbf z\in\mathcal Z(\boldsymbol\nu)$, assign its type-position
occurrences to eligible machines by solving
\begin{align}
    T_{\mathrm{tpl}}(\mathbf z):=
    \min_{\vartheta,\mathbf x^{\mathrm{tpl}}}\quad
    &\vartheta
    \label{eq:tctr_integer_template_ejor}\\
    \text{subject to}\quad
    &\sum_{M\in\mathcal E_{i,h}}x^{\mathrm{tpl}}_{i,h,M}=z_i,
    &&i=1,\ldots,R,\ h=1,\ldots,H_i,\notag\\
    &\sum_{i=1}^R
      \sum_{\substack{1\leq h\leq H_i\\M\in\mathcal E_{i,h}}}
      p_{i,h}(M)x^{\mathrm{tpl}}_{i,h,M}\leq\vartheta,
    &&M\in\mathcal M,\notag\\
    &x^{\mathrm{tpl}}_{i,h,M}\in\mathbb N_0.\notag
\end{align}
This model is the integral counterpart of
\eqref{eq:multitype_fluid_lp_ejor}: it preserves type-position balance and
minimizes the largest assigned machine workload, while job identities,
precedence, and finite start times are deferred to replay. Each feasible
incumbent therefore defines a machine-labeled multiset containing $z_i$
occurrences of every operation position of job type $i$.

For a fixed shop structure, the model has one count variable per
eligible type-position-machine triple, so its numbers of variables and
constraints are independent of the target multiplicity scale.

Expand the counts into symbols $(i,h,\widetilde M)$, where $\widetilde M$ is
the assigned template machine label. Deterministic orderings and circular
rotations produce a finite set of candidate streams from this multiset; the
streams contain the same symbols and differ only in circular order. The candidate-batch formulas,
ordering keys, rotations, tie-breaking rules, and reuse protocol used in the
experiments are specified in the online supplement.

\subsection{Replay, screening, and ECT refinement}
\label{subsec:tctr_template}

For each type-position pair $(i,h)$, let $Q_{i,h}^{\mathrm{ready}}$ be the FIFO
queue of job copies whose next unscheduled operation is at position
$h$. For a copy $j=J_{i,k}$, let $\tau_j$ be its ready time. Initially,
$Q_{i,1}^{\mathrm{ready}}$ contains $J_{i,1},\ldots,J_{i,\nu_i}$ in index
order, all later queues are empty, every $\tau_j=0$, and all machine intervals
are free.

Scan a stream $\pi$ circularly. At symbol $(i,h,\widetilde M)$, skip the
symbol when $Q_{i,h}^{\mathrm{ready}}$ is empty. Otherwise remove its first
copy $j=J_{i,k}$ and compute, for every eligible machine
$M\in\mathcal E_{i,h}$,
\begin{equation}
    s_{j,h,M}
    :=\min\left\{
        t\geq\tau_j:
        [t,t+p_{i,h}(M))\text{ is free on }M
    \right\}.
    \label{eq:tctr_earliest_gap_ejor}
\end{equation}

To screen a candidate stream, insert each operation on its template-specified
machine $\widetilde M$ at $s_{j,h,\widetilde M}$. After the best stream is
selected, one additional replay preserves its type-position order but chooses
an eligible machine by ECT machine selection, minimizing
$s_{j,h,M}+p_{i,h}(M)$. The deterministic ECT tie-breaking rule is specified
in the supplement.

After either insertion, set $\tau_j$ to the completion time of the inserted
operation and append $j$ to $Q^{\mathrm{ready}}_{i,h+1}$ unless $h=H_i$.
\enlargethispage{\baselineskip}
Replay stops after
\[
    N_{\mathrm{op}}(\boldsymbol\nu)
    :=\sum_{i=1}^R\nu_iH_i
\]
insertions. TCTR returns the better of the screened schedule and the schedule
obtained by ECT machine selection for the selected stream.

\begin{algorithm}[H]
\caption{Type-based cyclic template replay (TCTR)}
\label{alg:tctr_ejor}
\begin{algorithmic}[1]
\Require Fixed shop structure, multiplicities $\boldsymbol\nu$, finite
         candidate set $\mathcal Z(\boldsymbol\nu)$ with
         $z_i>0$ exactly when $\nu_i>0$
\Ensure Feasible completion-event schedule $\mathsf S^{\mathrm{TCTR}}$
\State $\mathcal C^{\mathrm{lab}}\gets\emptyset$
\ForAll{$\mathbf z\in\mathcal Z(\boldsymbol\nu)$}
    \State Optimize \eqref{eq:tctr_integer_template_ejor} and retain a feasible
           incumbent $\mathbf x^{\mathrm{tpl}}$
    \State Construct the finite set of candidate circular streams for
           $(\mathbf z,\mathbf x^{\mathrm{tpl}})$
    \ForAll{candidate stream $\pi$}
        \State Replay $\pi$ on the template-specified machines using
               earliest-gap insertion;
               add the result to $\mathcal C^{\mathrm{lab}}$
    \EndFor
\EndFor
\State Select the stream $\pi^\star$ with the smallest screened makespan
\State Replay the type-position order of $\pi^\star$ with ECT machine selection
       to obtain
       $\mathsf S^{\mathrm{ECT}}$
\State \Return the better of the screened schedule for $\pi^\star$ and
       $\mathsf S^{\mathrm{ECT}}$
\end{algorithmic}
\end{algorithm}

\subsection{Feasibility and computational scope}
\label{subsec:tctr_complexity}

\begin{proposition}[Feasibility of TCTR replay]
\label{prop:tctr_replay_feasibility_ejor}
Suppose $\mathcal Z(\boldsymbol\nu)$ is finite and nonempty, every retained
batch satisfies $z_i>0$ exactly when $\nu_i>0$, and
\eqref{eq:tctr_integer_template_ejor} has a feasible incumbent for each batch.
Algorithm~\ref{alg:tctr_ejor} terminates and
returns a feasible completion-event schedule for $\boldsymbol\nu$.
\end{proposition}

\begin{proof}
Template-specified and ECT-selected machines are eligible by construction, and earliest-gap
insertion preserves job precedence and machine capacity. After each insertion,
every unfinished job belongs to the queue of its next unscheduled position.
Hence, while work remains, at least one ready queue is nonempty. Because every
active type-position pair occurs in the stream, a complete scan schedules at
least one operation, so finitely many scans place all
$N_{\mathrm{op}}(\boldsymbol\nu)$ operations.

For any selected machine, the earliest feasible start is either $\tau_j$ or
the right endpoint of an occupied machine interval. Since $\tau_j$ is zero or
a predecessor completion, every inserted start is zero or an
operation-completion epoch. Proposition~\ref{prop:model_generated_ejor} then
gives a feasible completion-event schedule.
\end{proof}

Replay explicitly constructs all $N_{\mathrm{op}}(\boldsymbol\nu)$ operations
and can require at most that many complete stream scans, so schedule
construction is necessarily output sensitive. ECT adds only one replay of the
selected stream rather than one replay per candidate. Unlike the exact
attaining cycles of Section~\ref{sec:mmc}, the circular streams need not return
to a projected state; their finite feasibility follows from the replay procedure.

\section{Computational Study}
\label{sec:experiments}

The computational study complements the theory in two ways. First, a
controlled instance and a dense fixed-composition study examine how the recurrent
fluid structure manifests itself on finite horizons. Second, a broad public
benchmark and a nested component study evaluate TCTR as a finite-horizon
scheduling method and identify the mechanisms behind its performance. The
public studies draw on 169 base instances from the Barnes, Brandimarte, Dauzere,
and Hurink E/R/V families
\citep{mastrolilli2000,brandimarte1993,dauzere1997,hurink1994}.

\subsection{Experimental design and evaluation metrics}

Unless stated otherwise, TCTR follows Section~\ref{sec:tctr}; detailed batch
formulas, stream orders, rotations, tie-breaking rules, and reuse protocols are
reported in the supplement. The public studies solve the TCTR template model
with CP-SAT; a feasible incumbent suffices to define a TCTR template.

We compare TCTR with FIFO-EGI, SPT-EGI, and MWKR-EGI dispatching rules, an
independently initialized job-indexed adaptive large-neighborhood search
(J-ALNS), a job-indexed hybrid genetic algorithm (HGA), and an independent
job-indexed CP-SAT model where applicable. J-ALNS and HGA receive neither
fluid-allocation information nor TCTR information. Study-specific time budgets
are stated below.

The principal quality measure is the relative gap to the complete finite-demand
fluid lower bound,
\begin{equation}
    \operatorname{Gap}_{F}
    :=\frac{C_{\max}-\mathrm{LB}_F}{\mathrm{LB}_F}\times100\%.
    \label{eq:experimental_fluid_gap}
\end{equation}
Here $\mathrm{LB}_F=\rho_F(\boldsymbol\nu)$ is recomputed for every realized
finite multiplicity vector. Along
$\boldsymbol\nu(q)=q\mathbf b+\mathbf r$, $\rho_F(\mathbf b)$ is the
attainable recurrent slope, whereas
$\rho_F(q\mathbf b+\mathbf r)$ is the complete finite-demand benchmark used in
$\operatorname{Gap}_F$. The fluid allocation itself is not supplied to any
heuristic. All methods are run on the same workstation; additional
implementation details are given in the supplement.

\subsection{Exact recurrence versus finite-horizon performance}
\label{subsec:four_way_state_comparison}

The exact recurrent construction of Section~\ref{sec:mmc} is a structural
certificate rather than a finite-horizon optimum. We illustrate this distinction
on a controlled instance with one job type, three operation positions, two
machines, and
\begin{equation}
    \mathbf P=
    \begin{pmatrix}
        10&4\\
        5&3\\
        9&7
    \end{pmatrix}.
    \label{eq:mmc_tctr_state_instance}
\end{equation}
For $N=160$, $\rho_F(\mathbf e_1)=63/8$ and
$\mathrm{LB}_F=1260$. The exact slot-linked construction uses the
$(y,T)=(8,63)$ attaining cycle and gives
$C_{\max}^{\mathrm{slot}}=1309$. TCTR selects $z=8$ and gives
$C_{\max}^{\mathrm{TCTR}}=1267$, while both a job-type DP and an
independent job-indexed CP-SAT model certify
$C_{\max}^{*}=1263$.

Figure~\ref{fig:mmc_tctr_optimal_transition_graph} compares the projected-state
supports of these four schedules.

\begin{figure}[H]
    \centering
    \includegraphics[width=\textwidth,trim=0bp 2bp 1bp 0bp,clip]{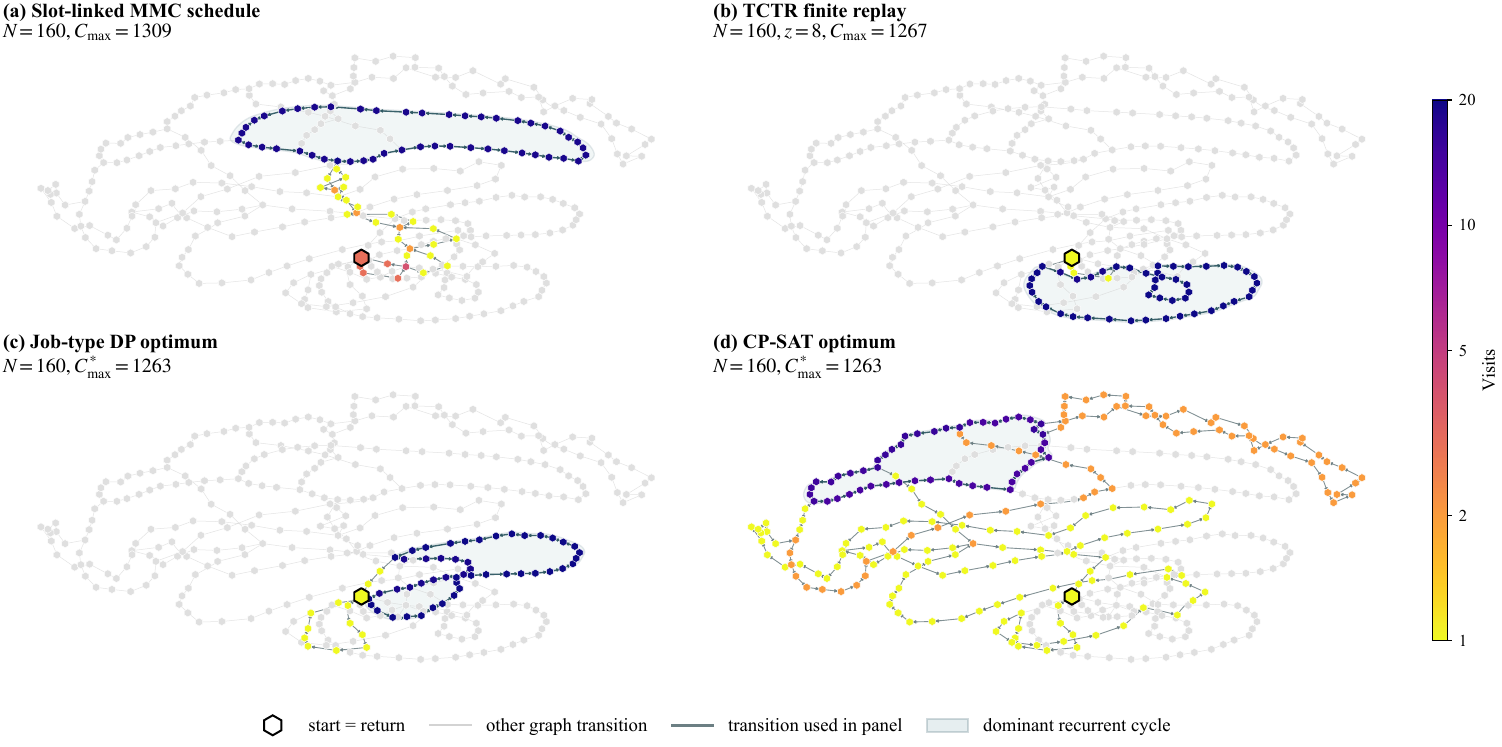}
    \setlength{\abovecaptionskip}{4pt}
    \caption{Projected-state supports for four schedules on the controlled
    instance. Panels share a 385-vertex, 409-transition union and one layout;
    positions have no state or time meaning. Color shows log-scaled arrivals,
    gray marks unused elements, shading identifies the dominant recurrent
    cycle, and the outlined hexagon marks the start/return state.}
    \label{fig:mmc_tctr_optimal_transition_graph}
\end{figure}

The comparison explains why finite replay remains useful after exact recurrent
attainment has been established. Exact cycle stitching is 46 time units above
the certified finite optimum, whereas TCTR is only four units above it. The
four schedules also follow markedly different projected-state trajectories.
Thus exact recurrent realization and near-optimal finite-horizon scheduling
need not coincide.

\subsection{Fixed-composition scaling toward the fluid regime}
\label{subsec:public_long_run_scaling}

The theory identifies $\rho_F(\mathbf b)$ as the attainable fluid workload
under fixed-composition growth. We therefore examine whether finite replay by job type shows
the same large-multiplicity tendency empirically.

The study uses 12 public base instances and two fixed-composition families per
base instance. The equal profile
uses $\mathbf b=\mathbf1,\mathbf r=\mathbf0$; the partial-growth profile grows
the longer half of the job types while keeping the shorter half at one copy.
The resulting 24 families are evaluated at 190 target sizes from $10^2$ to
$10^5$ operations. TCTR is compared with three dispatching rules and J-ALNS,
with J-ALNS receiving the larger of 10 seconds and the measured cold-start
TCTR runtime. All five methods return feasible schedules for all 4,560 cases.

\begin{figure}[H]
    \centering
    \includegraphics[width=\textwidth,trim=0bp 0bp 0bp 1bp,clip]{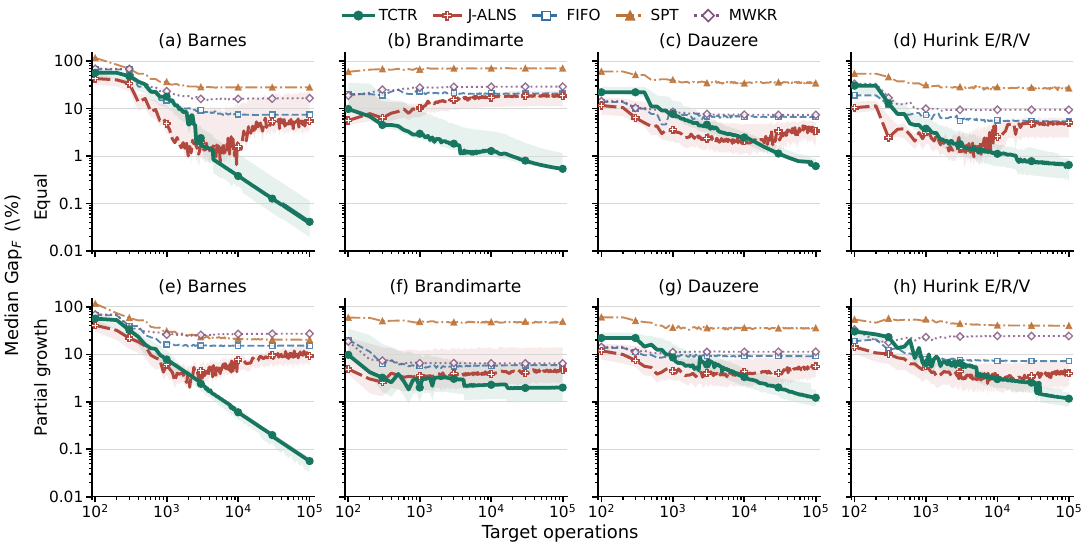}
    \setlength{\abovecaptionskip}{4pt}
    \caption{Family-resolved fluid-gap trends on the dense fixed-composition grid.
    Rows are profiles and columns are family groups; each panel summarizes
    three base instances, with interquartile bands for TCTR and J-ALNS.}
    \label{fig:public_dense_190_profile_trends}
\end{figure}

Figure~\ref{fig:public_dense_190_profile_trends} shows a systematic change with
multiplicity. At $10^3$ operations, J-ALNS has the lower median in six of eight
family-profile panels; by $10^4$, TCTR leads in seven. At $10^5$, TCTR
family medians range from 0.04\% to 0.65\% under equal growth and from 0.06\%
to 2.00\% under partial growth, while J-ALNS remains substantially farther
from the fluid bound. Hence job-indexed search remains competitive at moderate
sizes, but the value of repeated job-type structure becomes increasingly
visible as multiplicity grows. At $10^5$ operations, median runtimes are
378.75 seconds for TCTR and 379.51 seconds for J-ALNS.

Individual family trajectories and crossover statistics are reported in the
supplement.

\subsection{Broad finite-horizon benchmark}
\label{subsec:public_multiplicity_benchmark}

We next test whether the finite replay strategy remains effective beyond
selected fixed-composition families. The benchmark expands all 169 public base
instances under four deterministic multiplicity profiles, producing 676 cases with
9,581-10,356 operations and 450-2,016 jobs.

The seven methods are TCTR, the three dispatching baselines, J-ALNS, HGA, and
CP-SAT. TCTR template-allocation solves receive two seconds each; J-ALNS
receives 10 seconds per case. HGA uses a standard job-indexed FJSP encoding
\citep{zhang2011,rooyani2019} and reports the median of three deterministic
10-second runs. The independent CP-SAT model receives 300 seconds per case.
Neither HGA nor CP-SAT receives fluid or TCTR information.

\begin{figure}[H]
    \centering
    \includegraphics[width=\textwidth,trim=3bp 4bp 4bp 4bp,clip]{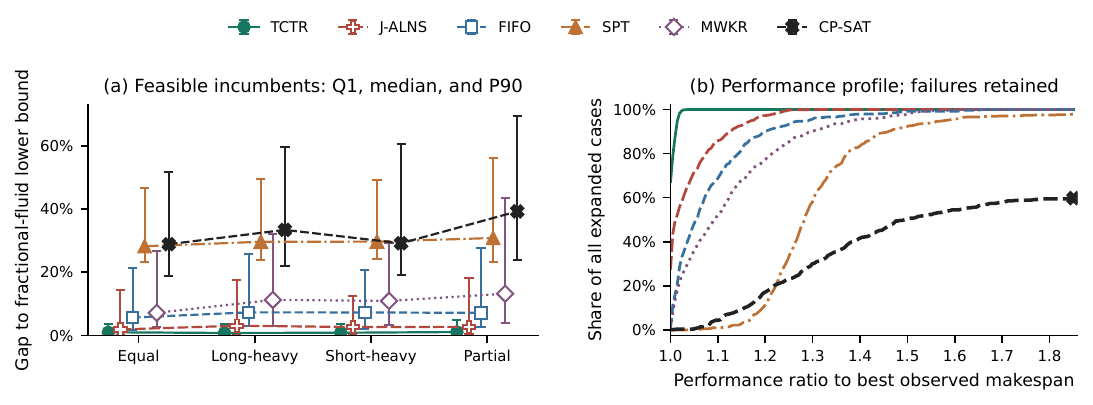}
    \setlength{\abovecaptionskip}{4pt}
    \caption{Benchmark on 676 expanded public cases. (a) Median fluid gaps
    with 25th-90th percentile whiskers among feasible incumbents. (b)
    Performance profiles relative to the best observed makespan. HGA is omitted
    from the panels for readability; in (b), it remains in the best-observed
    reference, and missing CP-SAT incumbents remain in the 676-case denominator.}
    \label{fig:public_multiplicity_benchmark}
\end{figure}

\begin{figure}[H]
    \centering
    \includegraphics[width=\textwidth,trim=4bp 5bp 4bp 3bp,clip]{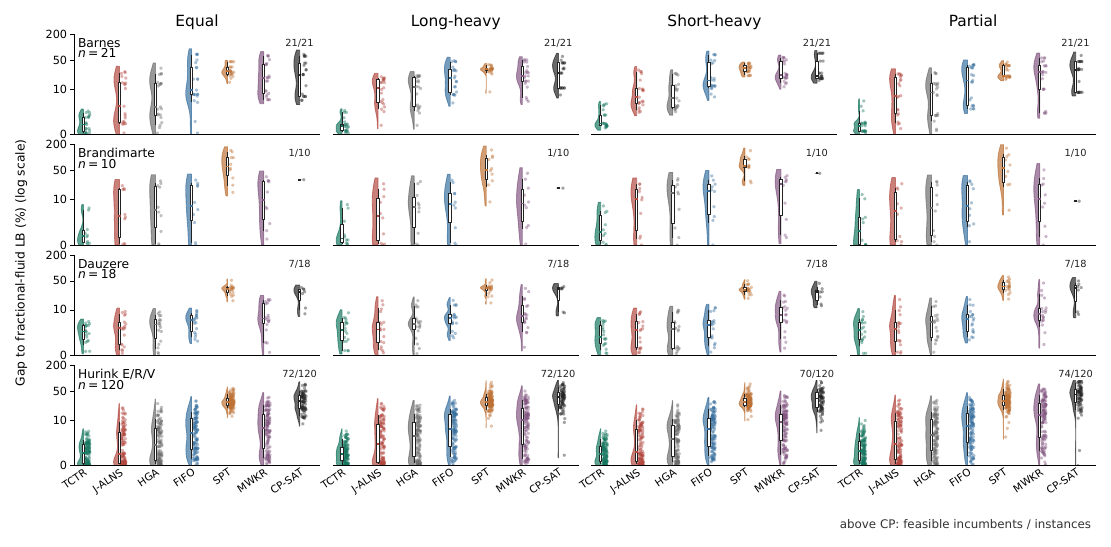}
    \setlength{\abovecaptionskip}{4pt}
    \caption{Fluid-gap distributions by family and multiplicity profile.
    Hurink E/R/V are pooled; half violins, boxes, and points show density,
    interquartile range with median, and individual cases on a log-spaced axis.
    CP-SAT distributions are conditional on a feasible incumbent.}
    \label{fig:public_multiplicity_raincloud}
\end{figure}

Figure~\ref{fig:public_multiplicity_benchmark} gives the main finite-horizon
result. TCTR, J-ALNS, and HGA are feasible on all 676 cases and achieve mean
fluid gaps of 1.54\%, 5.32\%, and 6.18\%, respectively. TCTR is within 2\%
of the best observed makespan across all seven methods on 98.5\% of the cases.
It records 456 wins, 6 ties, and 214 losses against J-ALNS, and 518 wins,
3 ties, and 155 losses against HGA.

CP-SAT returns an incumbent for 404 cases (59.8\%) and certifies one optimum;
its conditional mean fluid gap is 34.75\%. Full case-level results, paired
comparisons, solver statuses, and HGA replications are reported in the
supplement.

\begin{table}[H]
\centering
\caption{Family-profile results for the 676 expanded public cases.}
\label{tab:public_multiplicity_benchmark}
\begingroup
\footnotesize
\setlength{\tabcolsep}{1.0pt}
\renewcommand{\arraystretch}{0.94}
\begin{tabular*}{\textwidth}{@{\extracolsep{\fill}}llccrrrrrrrr@{}}
\toprule
Family & Profile & \(q\) range & \(N\) range & \(\overline{\mathrm{LB}}_F\) & \multicolumn{6}{c}{Mean \(C_{\max}\), complete methods} & CP-SAT inc. \\
\cmidrule(lr){6-11}\cmidrule(l){12-12}
 & & & & & TCTR & J-ALNS & HGA & FIFO & SPT & MWKR & mean \((f/a)\) \\
\midrule
Barnes & Equal & \([44,\,100]\) & \([660,\,1{,}005]\) & 51,559 & \textbf{52,035} & 56,522 & 56,683 & 64,158 & 67,381 & 65,020 & 67,139\,(21/21) \\
 & Long-heavy & \([19,\,45]\) & \([665,\,1{,}015]\) & 57,389 & \textbf{57,793} & 63,599 & 63,744 & 70,648 & 74,984 & 72,545 & 73,512\,(21/21) \\
 & Short-heavy & \([19,\,40]\) & \([665,\,1{,}015]\) & 48,891 & \textbf{49,639} & 54,029 & 53,595 & 62,820 & 64,797 & 63,958 & 66,718\,(21/21) \\
 & Partial & \([82,\,199]\) & \([663,\,1{,}000]\) & 62,709 & \textbf{63,227} & 69,377 & 68,236 & 75,039 & 80,356 & 79,431 & 80,904\,(21/21) \\
\addlinespace[2pt]
Brandimarte & Equal & \([42,\,182]\) & \([670,\,2{,}000]\) & 10,756 & \textbf{10,864} & 11,444 & 11,642 & 11,712 & 17,127 & 11,895 & 30,052\,(1/10) \\
 & Long-heavy & \([18,\,81]\) & \([660,\,1{,}980]\) & 12,387 & \textbf{12,508} & 12,866 & 13,084 & 13,201 & 18,843 & 13,393 & 31,977\,(1/10) \\
 & Short-heavy & \([18,\,74]\) & \([675,\,2{,}016]\) & 9,753 & \textbf{9,914} & 10,578 & 10,836 & 10,984 & 15,932 & 11,204 & 30,096\,(1/10) \\
 & Partial & \([76,\,344]\) & \([665,\,2{,}000]\) & 13,071 & \textbf{13,302} & 13,622 & 13,791 & 13,992 & 19,742 & 14,207 & 31,787\,(1/10) \\
\addlinespace[2pt]
Dauzere & Equal & \([26,\,51]\) & \([510,\,520]\) & 81,244 & \textbf{83,111} & 83,890 & 84,453 & 85,462 & 105,783 & 88,143 & 110,486\,(7/18) \\
 & Long-heavy & \([11,\,22]\) & \([455,\,495]\) & 83,870 & \textbf{86,243} & 87,200 & 88,195 & 89,379 & 109,998 & 92,209 & 115,772\,(7/18) \\
 & Short-heavy & \([11,\,22]\) & \([528,\,576]\) & 81,911 & \textbf{83,560} & 84,680 & 85,031 & 85,792 & 108,376 & 89,514 & 115,181\,(7/18) \\
 & Partial & \([46,\,89]\) & \([450,\,480]\) & 85,375 & \textbf{87,808} & 88,549 & 89,495 & 90,367 & 116,576 & 95,118 & 124,324\,(7/18) \\
\addlinespace[2pt]
Hurink E & Equal & \([33,\,200]\) & \([660,\,2{,}000]\) & 73,970 & \textbf{74,551} & 78,738 & 79,434 & 82,774 & 92,864 & 84,945 & 90,668\,(40/40) \\
 & Long-heavy & \([14,\,91]\) & \([665,\,2{,}002]\) & 82,371 & \textbf{83,119} & 89,891 & 89,735 & 93,251 & 104,310 & 97,996 & 103,202\,(40/40) \\
 & Short-heavy & \([14,\,80]\) & \([665,\,2{,}016]\) & 69,655 & \textbf{70,278} & 74,226 & 74,729 & 77,843 & 88,509 & 80,434 & 87,115\,(40/40) \\
 & Partial & \([66,\,399]\) & \([663,\,2{,}000]\) & 90,744 & \textbf{91,694} & 99,458 & 99,253 & 105,049 & 114,726 & 110,640 & 115,298\,(40/40) \\
\addlinespace[2pt]
Hurink R & Equal & \([33,\,200]\) & \([660,\,2{,}000]\) & 69,108 & \textbf{69,924} & 70,250 & 71,115 & 72,998 & 91,456 & 74,829 & 83,807\,(32/40) \\
 & Long-heavy & \([14,\,91]\) & \([665,\,2{,}002]\) & 76,091 & \textbf{77,035} & 78,568 & 79,957 & 82,222 & 100,866 & 85,467 & 97,055\,(32/40) \\
 & Short-heavy & \([14,\,80]\) & \([665,\,2{,}016]\) & 62,522 & \textbf{63,375} & 64,007 & 65,619 & 67,761 & 84,336 & 69,530 & 77,267\,(30/40) \\
 & Partial & \([66,\,399]\) & \([663,\,2{,}000]\) & 80,810 & \textbf{82,710} & 84,632 & 86,330 & 88,053 & 110,498 & 95,909 & 113,274\,(34/40) \\
\addlinespace[2pt]
Hurink V & Equal & \([33,\,200]\) & \([660,\,2{,}000]\) & 69,108 & 69,564 & \textbf{69,245} & 69,415 & 70,119 & 87,206 & 70,231 & -\,(0/40) \\
 & Long-heavy & \([14,\,91]\) & \([665,\,2{,}002]\) & 76,091 & 76,529 & \textbf{76,345} & 76,824 & 78,037 & 96,560 & 78,402 & -\,(0/40) \\
 & Short-heavy & \([14,\,80]\) & \([665,\,2{,}016]\) & 62,522 & 62,875 & \textbf{62,788} & 63,028 & 63,992 & 78,569 & 64,340 & -\,(0/40) \\
 & Partial & \([66,\,399]\) & \([663,\,2{,}000]\) & 80,810 & 81,445 & \textbf{81,292} & 82,047 & 83,008 & 105,391 & 84,804 & -\,(0/40) \\
\midrule
All families & All profiles & \([11,\,399]\) & \([450,\,2{,}016]\) & 69,270 & \textbf{70,118} & 72,447 & 72,973 & 75,661 & 90,088 & 78,222 & 92,194\,(404/676) \\
\addlinespace[2pt]
\multicolumn{5}{l}{Overall mean gap to \(\mathrm{LB}_F\) (\%)} & \textbf{1.54} & 5.32 & 6.18 & 10.05 & 33.59 & 14.12 & 34.75 \\
\multicolumn{5}{l}{Best/tied (\%)} & \textbf{67.3} & 27.4 & 8.3 & 0.9 & 0.0 & 0.4 & 0.1 \\
\multicolumn{5}{l}{Feasible (\%)} & \textbf{100.0} & \textbf{100.0} & \textbf{100.0} & \textbf{100.0} & \textbf{100.0} & \textbf{100.0} & 59.8 \\
\multicolumn{5}{l}{Median time (s)} & 7.009 & 10.036 & 10.035 & \textbf{0.025} & 0.126 & 0.038 & 300.606 \\
\bottomrule
\end{tabular*}

\vspace{2pt}

\parbox{\textwidth}{\scriptsize\emph{Notes.} The \(q\) and \(N\) columns
report ranges; \(\overline{\mathrm{LB}}_F\) and \(C_{\max}\) are rounded group
means. Bold marks the smallest group-mean makespan among complete-coverage
methods. HGA entries are medians of three independent 10-second runs. CP-SAT
incumbent and gap means are conditional on feasibility; \(f/a\) denotes
feasible/attempted cases, whereas rates use all 676 cases.}
\endgroup
\end{table}

Table~\ref{tab:public_multiplicity_benchmark} shows that the aggregate result is
not driven by a single family: TCTR has the smallest group-mean makespan among
complete-coverage methods in 20 of the 24 family-profile groups, with J-ALNS
leading in the four Hurink V groups. On Hurink V, the final ECT replay is
particularly important, reducing TCTR's mean fluid gap from 4.97\% before
refinement to 0.82\%.

Figure~\ref{fig:public_multiplicity_raincloud} further confirms that the
overall advantage persists across families and multiplicity profiles rather
than being generated by a small subset of cases.

\subsection{Contribution of TCTR components}

Finally, we evaluate five nested stages on the same 676 cases. Core replay uses
fixed-grid batches, one stream, append-only placement, and template labels.
Successive stages add earliest-gap insertion (EGI), the full set of candidate streams,
adaptive batches, and ECT refinement. Because each stage retains the preceding
alternatives, the comparisons measure conditional marginal gains along the
implemented TCTR pipeline.

\begin{table}[H]
\centering
\caption{Nested TCTR build-up on 676 expanded cases. Family entries are mean
fluid gaps (Hurink pools E/R/V). P90 pools all cases; ``Improved'' and median
reduction compare consecutive stages, with reductions conditional on
improvement. A down arrow marks a relative gap decrease of at least 25\%.}
\label{tab:tctr_component_build_up}
\scriptsize
\setlength{\tabcolsep}{2.2pt}
\begin{tabular*}{\textwidth}{@{\extracolsep{\fill}}lccccrrrrrrr}
\toprule
& \multicolumn{4}{c}{Components}
& \multicolumn{4}{c}{Mean fluid gap by family}
& \multicolumn{3}{c}{All cases} \\
\cmidrule(lr){2-5}\cmidrule(lr){6-9}\cmidrule(lr){10-12}
Stage & EGI & Streams & Adapt. & ECT
& \multicolumn{1}{c}{Barnes} & \multicolumn{1}{c}{Brand.}
& \multicolumn{1}{c}{Dauzere} & \multicolumn{1}{c}{Hurink}
& \multicolumn{1}{c}{P90} & \multicolumn{1}{c}{Improved}
& \multicolumn{1}{c}{Med. reduction} \\
\midrule
Core replay & - & - & - & -
& $30.42\%$ & $22.14\%$ & $36.44\%$ & $29.00\%$
& $58.88\%$ & - & - \\
$+$ gap insertion & \checkmark & - & - & -
& $2.63\%$\rlap{\hspace{1.5pt}$\downarrow$}
& $3.36\%$\rlap{\hspace{1.5pt}$\downarrow$}
& $7.47\%$\rlap{\hspace{1.5pt}$\downarrow$}
& $5.72\%$\rlap{\hspace{1.5pt}$\downarrow$}
& $12.87\%$\rlap{\hspace{1.5pt}$\downarrow$}
& $663$ ($98.1\%$) & $18.75\%$ \\
$+$ streams & \checkmark & \checkmark & - & -
& $1.21\%$\rlap{\hspace{1.5pt}$\downarrow$}
& $2.47\%$\rlap{\hspace{1.5pt}$\downarrow$}
& $5.20\%$\rlap{\hspace{1.5pt}$\downarrow$}
& $3.94\%$\rlap{\hspace{1.5pt}$\downarrow$}
& $9.61\%$\rlap{\hspace{1.5pt}$\downarrow$}
& $566$ ($83.7\%$) & $1.25\%$ \\
$+$ adaptive & \checkmark & \checkmark & \checkmark & -
& $1.16\%$ & $2.37\%$ & $5.01\%$ & $3.73\%$
& $9.25\%$ & $148$ ($21.9\%$) & $0.29\%$ \\
TCTR & \checkmark & \checkmark & \checkmark & \checkmark
& $1.04\%$ & $2.19\%$
& $2.86\%$\rlap{\hspace{1.5pt}$\downarrow$}
& $1.37\%$\rlap{\hspace{1.5pt}$\downarrow$}
& $3.97\%$\rlap{\hspace{1.5pt}$\downarrow$}
& $441$ ($65.2\%$) & $1.70\%$ \\
\bottomrule
\end{tabular*}
\end{table}

Table~\ref{tab:tctr_component_build_up} shows that the largest marginal gain
comes from earliest-gap insertion, followed by stream diversification.
Adaptive batches act selectively, while the final ECT replay still improves
441 of 676 cases. The results reinforce the role assigned to TCTR in the
paper: the recurrent theory supplies a compact repeated job-type structure,
whereas finite replay, diversification, and machine-assignment adaptation determine how
effectively that structure is used on a particular finite demand.

\FloatBarrier

\section{Conclusion}
\label{sec:conclusion}

This paper establishes a direct connection between two complementary views of
high-multiplicity flexible job shop scheduling: workload allocation through a
fluid relaxation and repeated execution through a recurrent discrete schedule.
The main theoretical result shows that, after a suitable finite scaling, an
optimal fluid allocation can be realized exactly by a feasible recurrent
discrete schedule with the same workload value. Period-shifted slot linking
provides the constructive bridge by preserving the machine schedule while
grouping operation occurrences across repetitions into jobs that satisfy their
required operation order. This recurrent attainment also has a finite-instance
consequence: under fixed-composition growth, the optimal makespan remains
within a constant additive gap of the fluid lower bound, so the relative gap
vanishes as the number of job copies grows.

The recurrent result also changes how finite high-multiplicity instances can
be approached computationally. Rather than treating every job copy as an
independent scheduling object, it reveals that much of the relevant structure
can be represented through compact job-type patterns and realized only when a
finite schedule is constructed. TCTR carries this idea into finite-horizon
scheduling without requiring exact recurrence, allowing the repeated pattern
and machine assignments to adapt to the actual finite instance. Computational
results support this translation from theory to scheduling practice: across
676 multiplicity-expanded public FJSP instances, TCTR produces a feasible
schedule in every case with a 1.54\% mean gap to the fluid lower bound.

The analysis therefore suggests a broader role for fluid relaxations in
high-multiplicity scheduling. Beyond providing lower bounds or guidance for
subsequent search, an optimal fluid allocation can reveal discrete recurrent
structure that is useful both for theoretical performance guarantees and for
finite scheduling decisions. Extending this connection to settings with
setup times, transport times, blocking, uncertainty, or other shop constraints
is a natural direction for future research.

\bibliographystyle{elsarticle-harv}
\bibliography{refs}

\end{document}


\begin{frontmatter}
\title{Online Supplement to ``High-Multiplicity Flexible Job Shops: From Exact
Recurrent Fluid Attainment to Structure-Guided Finite-Horizon Scheduling''}
\author[sds]{Wenjun Zheng}
\ead{wenjunzheng@link.cuhk.edu.cn}
\author[sds]{Wei Qu}
\ead{weiqu@link.cuhk.edu.cn}
\author[sse]{Weilin Cai}
\ead{weilincai@link.cuhk.edu.cn}
\author[sds,sribd]{Jianfeng Mao\corref{cor1}}
\ead{jfmao@cuhk.edu.cn}
\cortext[cor1]{Corresponding author}

\address[sds]{School of Data Science, The Chinese University of Hong Kong,
Shenzhen (CUHK-Shenzhen), Guangdong 518172, China}
\address[sse]{School of Science and Engineering, The Chinese University of
Hong Kong, Shenzhen (CUHK-Shenzhen), Guangdong 518172, China}
\address[sribd]{Shenzhen Research Institute of Big Data, Guangdong 518172,
China}
\end{frontmatter}

\section{Worked single-job-type specialization}

This supplement develops the single-job-type case as a self-contained example.
It separates the three ideas that are combined across job types in the main paper:
externalizing unstarted supply, linking packed slots through period offsets,
and translating unfinished generations between consecutive boundaries. The
general fixed-composition theorem in the main paper is proved independently and does
not rely on this supplement.

\begin{example}[One repeated job type]
\label{ex:single_route_supp}
Fix one job type whose operation sequence has $H$ positions, and consider $N$
interchangeable copies of this type. Let
$X=(\mathbf n,\mathbf m,\mathcal W)$ be the completion-event state: $n_h$
counts idle jobs whose next operation is at position $h$, $\mathbf m$ records
idle machines, and $\mathcal W$ is the multiset of active operations.
\end{example}

\subsection{Projected recurrence and finite connectors}

Only $n_1$ records jobs that have not begun. Delete that coordinate and admit a
fresh job whenever a position-$1$ operation is assigned:
\begin{equation}
    \operatorname{proj}(X)
    :=\bigl((n_2,\ldots,n_H)^\top,\mathbf m,\mathcal W\bigr).
    \label{eq:single_projection_supp}
\end{equation}
All in-process jobs and the complete machine state remain visible.

The weighted projected graph
$\mathcal G_\infty=(V_\infty,E_\infty,w)$ has three components. Its ordinary
vertices are projected states reachable from
$v_0=(\mathbf0,\bar{\mathbf m},\emptyset)$ under unlimited external position-$1$
supply. Every admissible assignment gives a zero-weight edge; a position-$1$
assignment is also an admission edge. If one temporal-advancement action of
duration $\Delta$ completes $d$ descriptors, serialize those completions in a
fixed canonical order through a forced chain of $d$ edges. The first $d-1$
edges have weight zero and the last has weight $\Delta$.

The graph is generally infinite because unfinished-job inventories are not
truncated. A finite walk nevertheless lifts to a feasible execution whenever a
distinct job is supplied at each admission edge. Thus projection can represent
positive net production by a state-returning closed walk without removing a
scheduling constraint. Figure~\ref{fig:single_projection_supp} gives the
smallest example.

\begin{figure}[htbp]
    \centering
    \includegraphics[width=\textwidth,trim=4bp 5bp 2bp 4bp,clip]{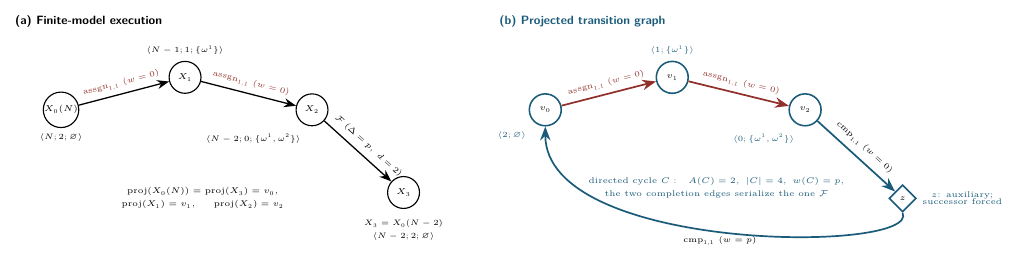}
    \setlength{\abovecaptionskip}{4pt}
    \caption{External-supply projection for a one-operation job type and two
    identical machines. Two assignments and one temporal-advancement step
    transform $X_0(N)$ into $X_0(N-2)$, but their projected endpoints coincide.
    Serializing the simultaneous completions produces a finite closed walk.}
    \label{fig:single_projection_supp}
\end{figure}

For operation position $h$, let $p_h(M)$ be its processing time on eligible
machine $M\in\mathcal E_h$. The fluid workload per completed job is
\begin{equation}
\begin{aligned}
    \rho_F:=\min_{\rho,\lambda}\quad &\rho\\
    \text{subject to}\quad
    &\sum_{M\in\mathcal E_h}\lambda_{h,M}=1,
    &&h=1,\ldots,H,\\
    &\sum_{h:M\in\mathcal E_h}p_h(M)\lambda_{h,M}\leq\rho,
    &&M\in\mathcal M,\\
    &\lambda_{h,M}\geq0.
\end{aligned}
\label{eq:single_fluid_lp_supp}
\end{equation}

\begin{lemma}[Single-job-type connectors]
\label{lem:single_connectors_supp}
Every ordinary vertex $v$ has a finite filling path from $v_0$ to $v$ and a
finite no-admission draining path from $v$ to $v_0$. If $v$ is all-idle with
$\beta_h$ jobs waiting at position $h=2,\ldots,H$, these paths can be chosen
serially so that
\begin{equation}
    T_v^{\mathrm{in}}+T_v^{\mathrm{out}}
    =\beta P,
    \qquad
    \beta:=\sum_{h=2}^H\beta_h,
    \qquad
    P:=\sum_{h=1}^H\min_{M\in\mathcal E_h}p_h(M).
    \label{eq:single_connector_identity_supp}
\end{equation}
\end{lemma}

\begin{proof}
The filling path exists by reachability. To drain, complete every active
operation and then process the finitely many waiting jobs through their
remaining positions without further admissions. For an all-idle state, create
each unfinished job by processing its completed operation-sequence prefix serially and drain
it by processing its remaining suffix serially, always using a fastest eligible
machine. The two pieces contain every operation of that job exactly once.
\end{proof}

For later use, let $C$ denote any nonempty finite closed walk, with admission
count $A(C)$, elapsed weight $w(C)$, and edge count $|C|$.
Every feasible $N$-job schedule satisfies $C_{\max}\geq N\rho_F$: its
machine-assignment counts divided by $N$ are feasible in
\eqref{eq:single_fluid_lp_supp} with $\rho=C_{\max}/N$. The connectors will
later place any recurrent core between finite entry and exit pieces. Their
amortization is previewed in Figure~\ref{fig:single_amortization_supp}.

\begin{figure}[htbp]
    \centering
    \includegraphics[width=0.98\textwidth,trim=2bp 5bp 2bp 2bp,clip]{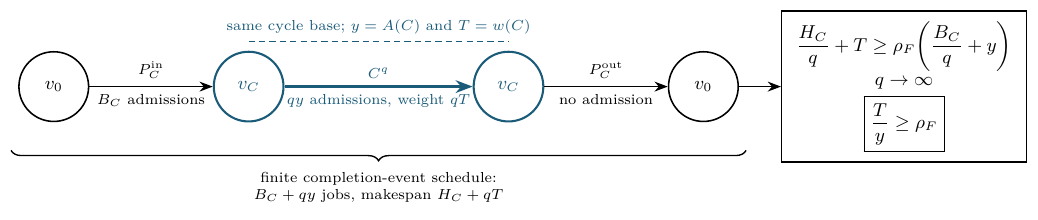}
    \setlength{\abovecaptionskip}{4pt}
    \caption{Connector amortization. Fixed paths turn $C^q$ into a finite
    schedule; after division by $q$, their contributions vanish and leave
    $w(C)/A(C)\geq\rho_F$.}
    \label{fig:single_amortization_supp}
\end{figure}

\subsection{Period-shifted slot linking}

\begin{lemma}[Single-job-type period-shifted slot linking]
\label{lem:single_slot_linking_supp}
Let $\lambda$ be any rational fluid allocation satisfying the position balances in
\eqref{eq:single_fluid_lp_supp}, and define
\[
    \rho(\lambda):=
    \max_{M\in\mathcal M}
    \sum_{h:M\in\mathcal E_h}p_h(M)\lambda_{h,M}.
\]
There exist $y\in\mathbb N_+$, $T=y\rho(\lambda)$, and a feasible
$T$-periodic completion-event schedule that admits $y$ jobs per period and
returns at every period boundary to the same finite all-idle projected state.
One period induces a finite closed walk $C_\lambda$ with
\begin{equation}
    A(C_\lambda)=y,
    \qquad w(C_\lambda)=T,
    \qquad |C_\lambda|=2Hy.
    \label{eq:single_linking_statistics_supp}
\end{equation}
\end{lemma}

\begin{proof}
\emph{Integerize and pack.}
Choose $y$ so that
$a_{h,M}:=y\lambda_{h,M}$ and $T:=y\rho(\lambda)$ are integral. On each
machine, pack its $a_{h,M}$ unlabelled position-$h$ occurrences consecutively
in $[0,T)$. The resulting ordered intervals form a period-$T$ slot pattern,
repeated every $T$ time units. There are exactly $y$ slots at every operation
position, and a slot on at least one maximum-load machine ends at $T$.

\emph{Link by period offsets.}
Label the slots at each position by lanes $k=1,\ldots,y$. Let $M_{h,k}$ and
$\varphi_{h,k}$ be the machine and phase of slot $(h,k)$. Set $z_{1,k}=0$ and
recursively choose
\begin{equation}
z_{h+1,k}:=\min\left\{
u\in\mathbb Z_{\geq z_{h,k}}:
\varphi_{h+1,k}+uT
\geq\varphi_{h,k}+z_{h,k}T+p_h(M_{h,k})
\right\}.
\label{eq:single_offset_recursion_supp}
\end{equation}
Every slot ends by $T$, so
$z_{h+1,k}-z_{h,k}\in\{0,1\}$ and $0\leq z_{h,k}\leq h-1$.
Job $J_{k,s}$ uses period copy $s+z_{h,k}$ of slot $(h,k)$ and starts at
\[
    s_{h,k,s}:=\varphi_{h,k}+(s+z_{h,k})T.
\]
For fixed $(h,k)$, the map $s\mapsto s+z_{h,k}$ is a bijection on
$\mathbb Z$, so every machine slot is used exactly once. The offset recursion
enforces job precedence without moving a slot.

\begin{figure}[htbp]
    \centering
    \includegraphics[width=\textwidth,trim=4bp 3bp 2bp 4bp,clip]{figure/theorem1_periodic_mmc_construction_v13.pdf}
    \setlength{\abovecaptionskip}{4pt}
    \caption{Single-job-type lane construction for $H=3$, $y=2$, and $T=4$.
    Integer counts form fixed per-machine slot patterns; period-copy
    relabelling enforces
    precedence; translating unfinished generations gives identical consecutive
    boundary states.}
    \label{fig:single_linking_supp}
\end{figure}

\emph{Translate unfinished generations.}
Let $c_{h,k,s}$ be the completion time of position $h$ of $J_{k,s}$. At
boundary $nT$, after boundary completions and before phase-zero assignments,
define
\[
    \eta_{k,s}(n):=\min\bigl(\{h:c_{h,k,s}>nT\}\cup\{H+1\}\bigr)
\]
for admitted generations $s\leq n-1$, and
\[
    \beta_h(n):=\sum_{k=1}^y\sum_{s\leq n-1}
    \mathbf{1}\{\eta_{k,s}(n)=h\},
    \qquad h=2,\ldots,H.
\]
Only finitely many generations contribute because $z_{h,k}\leq h-1$.
Moreover, $c_{h,k,s+1}=c_{h,k,s}+T$, so
$\eta_{k,s+1}(n+1)=\eta_{k,s}(n)$ and every $\beta_h(n)$ repeats. All machines are
idle and no operation is active at either boundary. The projected state is
therefore finite and identical. Lemma~\ref{lem:single_connectors_supp} makes it
reachable and drainable, so one period is a finite closed walk. Its edge count
is $2Hy$ because the period contains $y$ assignments and $y$ completions at
each of the $H$ positions.
\end{proof}

\subsection{Exact recurrent fluid attainment}

\begin{theorem}[Single-job-type exact recurrent fluid attainment]
\label{thm:single_attainment_supp}
For one repeated job type, there exist $y\in\mathbb N_+$ and a feasible
$y\rho_F$-periodic completion-event schedule that admits $y$ jobs per period
and returns to the same finite all-idle projected state. One period induces a
finite closed walk $\mathcal C_F$ satisfying
\[
    A(\mathcal C_F)=y,
    \qquad |\mathcal C_F|=2Hy,
    \qquad w(\mathcal C_F)=y\rho_F.
\]
\end{theorem}

\begin{proof}
A rational optimum $\lambda^*$ exists in
\eqref{eq:single_fluid_lp_supp}. Lemma~\ref{lem:single_slot_linking_supp}
constructs the stated periodic schedule and finite closed walk. The edge count
follows because state return balances assignments and completions separately
at every operation position, giving $y$ of each per period.
\end{proof}

\begin{remark}
The bi-infinite generation index describes a steady periodic realization, not
an infinite finite-instance inventory. At any boundary only finitely many jobs
are unfinished; negative generations represent precisely that finite initial
unfinished-job inventory. The filling-period-draining construction in the main paper
provides explicit entry and exit paths.
\end{remark}

\subsection{Admission-normalized closed-walk consequence}

Let $\mathfrak C$ be the nonempty closed walks of $\mathcal G_\infty$. For
$C\in\mathfrak C$, let $A(C)$ be its number of admissions, $|C|$ its edge
count, and $w(C)$ its total weight. The following accounting identity connects
the admission-normalized and edge-normalized values.

\begin{lemma}[Single-job-type cycle accounting]
\label{lem:single_cycle_accounting_supp}
Every $C\in\mathfrak C$ has $A(C)>0$ and contains exactly $A(C)$ assignments
and completions at every operation position. Consequently,
\begin{equation}
    |C|=2HA(C),
    \qquad
    \frac{w(C)}{|C|}
    =\frac{1}{2H}\frac{w(C)}{A(C)}.
    \label{eq:single_cycle_accounting_supp}
\end{equation}
\end{lemma}

\begin{proof}
Rotate the walk to an ordinary vertex. Let $I_h$ and $F_h$ count its
position-$h$ assignments and completions. Recurrence at position $1$ gives
$F_1=I_1=A(C)$. For $h\geq2$, recurrence of the idle-plus-active population
gives $F_h=F_{h-1}$, while recurrence of the active population gives
$I_h=F_h$. Hence every count equals $A(C)$. If $A(C)=0$, no assignment or
completion edge remains, contradicting nonemptiness. Summing the counts gives
the edge identity and the relation between the two normalizations.
\end{proof}

The finite-schedule fluid lower bound also applies to every closed walk after
connector amortization. Place $C^q$ between its fixed filling and draining
paths. If the connector admissions and total weight are $B_C$ and $K_C$, the
lifted schedule has $B_C+qA(C)$ jobs and makespan $K_C+qw(C)$. Dividing the
fluid lower bound by $q$ and taking $q\to\infty$ gives
\begin{equation}
    \frac{w(C)}{A(C)}\geq\rho_F.
    \label{eq:single_cycle_lower_bound_supp}
\end{equation}
Theorem~\ref{thm:single_attainment_supp} supplies a closed walk attaining
equality. Therefore
\[
    \inf_{C\in\mathfrak C}\frac{w(C)}{A(C)}=\rho_F,
    \qquad
    \inf_{C\in\mathfrak C}\frac{w(C)}{|C|}=\frac{\rho_F}{2H},
\]
and both infima are attained by $\mathcal C_F$.

\section{TCTR implementation details}

The main paper defines the finite-horizon TCTR architecture: candidate batches,
circular priority streams, screening on template-specified machines, and one
ECT machine-selection refinement. This section specifies the deterministic
candidate-generation and selection rules used in the experiments, including
candidate batches, stream orders, rotations, tie breaking, and the
supplementary reuse protocol. As in Section~5 of the main paper, circular
stream wrap-around is a priority mechanism and does not impose projected-state
return.

\subsection{Candidate batches}

For target multiplicities $\boldsymbol\nu$, let
$I_+:=\{i:\nu_i>0\}$ and $m_+:=|I_+|$. Write
$v_{(1)}\leq\cdots\leq v_{(m_+)}$ for the positive components of
$\boldsymbol\nu$ and set $v_{\max}:=v_{(m_+)}$. The operator
$\operatorname{nint}(x)$ denotes nearest-integer rounding, with half-integer
ties sent to the even integer.

A trial value $\hat q$ represents a prospective number of stream repetitions.
For component cap $\zeta$, define the candidate batch
\begin{equation}
    z_i(\hat q;\zeta):=
    \begin{cases}
        0, & \nu_i=0,\\
        \max\{1,\min\{\zeta,
        \operatorname{nint}(\nu_i/\hat q)\}\}, & \nu_i>0.
    \end{cases}
    \label{eq:tctr_compressed_batch_supp}
\end{equation}
By construction, $z_i>0$ exactly when $\nu_i>0$.

The fixed-grid layer samples multiplicity extremes, quartiles, and dyadic scales.
For $f\in\{1/4,1/2,3/4\}$, let
$Q_f(\boldsymbol\nu):=
v_{\left(1+\operatorname{nint}((m_+-1)f)\right)}$, and define
\begin{equation}
\begin{aligned}
\mathcal Q_{\mathrm{FG}}(\boldsymbol\nu)
:=\operatorname{sortuniq}\Bigl(&
\{1,v_{(1)},v_{\max},Q_{1/4},Q_{1/2},Q_{3/4},\\
&\max(1,\lfloor v_{\max}/2\rfloor),
 \max(1,\lfloor v_{\max}/4\rfloor),
 \max(1,\lfloor v_{\max}/8\rfloor)\}\\
&\cup\{v_{(j)}:v_{(j)}\geq
              \max(1,\lfloor v_{\max}/4\rfloor)\}\Bigr).
\end{aligned}
\label{eq:tctr_fixed_grid_supp}
\end{equation}

The adaptive layer permits larger candidate batches as the target grows. Its
component budget is
\begin{equation}
\bar z_A(\boldsymbol\nu)
:=\min\left\{z_{\max},
    \max\left\{z_{\min},
      \left\lceil\log_2(1+\|\boldsymbol\nu\|_1)+\kappa_q\right\rceil
    \right\}\right\}.
\label{eq:tctr_adaptive_budget_supp}
\end{equation}
For $d=1,\ldots,\bar z_A(\boldsymbol\nu)$, define
\begin{equation}
\begin{aligned}
\mathcal Q_d(\boldsymbol\nu):=\Bigl\{
&\max(1,\lfloor v_{\max}/d\rfloor),
 \max(1,\operatorname{nint}(v_{\max}/d)),\\
&\max(1,\lceil v_{\max}/d\rceil)\Bigr\},
\qquad
\mathcal Q_A(\boldsymbol\nu)
:=\operatorname{sortuniq}
  \bigcup_{d=1}^{\bar z_A(\boldsymbol\nu)}
  \mathcal Q_d(\boldsymbol\nu).
\end{aligned}
\label{eq:tctr_adaptive_scales_supp}
\end{equation}

Let $\bar z_{\mathrm{FG}}$ and $\overline Z_{\mathrm{FG}}$ be the component and
total-size caps of the fixed-grid layer, and set
$\overline Z_A(\boldsymbol\nu):=
\max\{m_+,\kappa_Z\bar z_A(\boldsymbol\nu)\}$. The candidate sets are
\begin{equation}
\begin{aligned}
\mathcal Z_{\mathrm{FG}}(\boldsymbol\nu)
&:=\left\{
\mathbf z(\hat q;\bar z_{\mathrm{FG}}):
\hat q\in\mathcal Q_{\mathrm{FG}}(\boldsymbol\nu),\quad
\|\mathbf z\|_1\leq\overline Z_{\mathrm{FG}}
\right\},\\
\mathcal Z_A(\boldsymbol\nu)
&:=\left\{
\mathbf z(\hat q;\bar z_A(\boldsymbol\nu)):
\hat q\in\mathcal Q_A(\boldsymbol\nu),\quad
\|\mathbf z\|_1\leq\overline Z_A(\boldsymbol\nu)
\right\},\\
\mathcal Z(\boldsymbol\nu)
&:=\operatorname{sortuniq}
\left(\mathcal Z_{\mathrm{FG}}(\boldsymbol\nu)
      \cup\mathcal Z_A(\boldsymbol\nu)\right).
\end{aligned}
\label{eq:tctr_candidate_batches_supp}
\end{equation}
Fixed-grid batches precede adaptive batches before deduplication. If the union
is empty, TCTR uses $z_i=\mathbb I(\nu_i>0)$. The fixed-grid-only stage in the component build-up
restricts \eqref{eq:tctr_candidate_batches_supp} to
$\mathcal Z_{\mathrm{FG}}(\boldsymbol\nu)$; it is an intermediate stage rather
than a separate method.

The experiments use $\bar z_{\mathrm{FG}}=12$,
$\overline Z_{\mathrm{FG}}=50$, $z_{\min}=4$, $z_{\max}=24$,
$\kappa_q=2$, and $\kappa_Z=4$. Before deduplication, the fixed-grid and adaptive
layers contribute at most $m_++9$ and $3\bar z_A(\boldsymbol\nu)$ trial values,
respectively.

\subsection{Candidate orderings and deterministic selection}

For batch $\mathbf z$, let
$K_{\mathrm{tpl}}(\mathbf z):=\sum_{i=1}^Rz_iH_i$ be the number of packed
occurrences. Each occurrence records its type $i$, operation position $h$,
template machine label $M$, copy index $k$, packed phases $(\varphi,\psi)$,
and within-machine index $d$. The resulting machine-labeled multiset is
linearized by the following lexicographic keys:
\begin{align}
    \text{diagonal:}\quad
      &(k+h,\ h,\ \varphi,\ M,\ i),\label{eq:tctr_order_diagonal_supp}\\
    \text{template-start:}\quad
      &(\varphi,\ \psi,\ M,\ d,\ i,\ h,\ k),\label{eq:tctr_order_start_supp}\\
    \text{copy-stage:}\quad
      &(k,\ h,\ \varphi,\ M,\ i).
      \label{eq:tctr_order_copy_supp}
\end{align}
The diagonal key follows a copy-position wavefront, the template-start key
follows packed machine phases, and the copy-stage key aligns comparable copy
positions. All three retain the same machine-labeled multiset.

For rotation budget $R_{\mathrm{rot}}$, the starting indices are
\begin{equation}
    \mathcal D_{\mathrm{rot}}
    :=\operatorname{sortuniq}
      \left\{
      \operatorname{nint}\!\left(
      \frac{jK_{\mathrm{tpl}}}{R_{\mathrm{rot}}}\right)
      \bmod K_{\mathrm{tpl}}
      :j=0,\ldots,R_{\mathrm{rot}}-1
      \right\}.
    \label{eq:tctr_rotations_supp}
\end{equation}
The experiments use $R_{\mathrm{rot}}=8$. The three orderings and these
rotations form the finite set of candidate circular streams for each batch and
its integer machine assignments.

Screening on the template-specified machines ranks candidates
lexicographically by makespan, $\|\mathbf z\|_1$, trial value $\hat q$,
candidate index, ordering index, and rotation index. Candidate indexing places
fixed-grid batches before adaptive batches. After the best screened stream is
selected, ECT machine selection chooses the machine for the next operation of
job copy $j=J_{i,k}$ by
\begin{equation}
    M_{j,h}^{\mathrm{ECT}}
:=\underset{M\in\mathcal E_{i,h}}{\operatorname{arg\,lex\,min}}
    \bigl(s_{j,h,M}+p_{i,h}(M),\ s_{j,h,M},\ p_{i,h}(M),\ M\bigr).
\label{eq:tctr_ect_machine_selection_supp}
\end{equation}
Thus ECT first minimizes completion time, followed by start time, processing
time, and machine index. The ECT-refined schedule replaces the screened
schedule only when its makespan is strictly smaller; a tie retains the
screened schedule.

\subsection{Direct replay and calibrated reuse}

All results reported in the main paper, including the 1.54\% mean fluid gap,
use direct per-target replay. Direct experiments construct
$\mathcal Z(\boldsymbol\nu)$ and screen the full set of candidate streams separately for
every target vector. Integer template solutions are cached by $\mathbf z$
whenever several targets share the same fixed shop structure; the target
still performs its own stream screening and final ECT comparison.
In calibrated-reuse experiments, a prespecified vector
$\boldsymbol\nu^{\mathrm{cal}}$ selects one stream and replay realization. That
choice is then replayed on further multiplicity vectors with the same active job
types. Because $z_i>0$ exactly when $\nu_i>0$, the progress argument used in the
main paper continues to hold, and replay proceeds until every operation of the
new target has been inserted. Calibrated reuse is supplementary only and
contributes to none of the main-paper results.

\section{Experimental protocols and comparator implementations}

This section records the common implementation choices omitted from the
condensed computational section of the main paper. All reported times are
wall-clock times and include model construction and decoding. Deterministic
case seeds are obtained by hashing the method version, data set, instance,
multiplicity profile, and, where applicable, replication index.

\subsection{TCTR and dispatching baselines}

For every target, TCTR generates the fixed-grid and adaptive batches specified
in Section~S2.1, solves the TCTR template model for each previously unseen
candidate batch with CP-SAT,
and screens the three stream orders and eight rotations directly on that
target. Each template solve has a 2-second limit and one solver worker; a
feasible incumbent is sufficient. Template solutions are cached by the pair
consisting of the fixed shop structure and batch vector, but replay scores
are not transferred between targets. The selected stream receives exactly one
ECT replay, and the better of its screened schedule and ECT schedule is
reported. Cold-start TCTR time charges every template solve needed by the
target together with all candidate replays.

FIFO-EGI orders ready operations by queue-release time, SPT-EGI by the
shortest eligible processing time of the next operation, and MWKR-EGI by
decreasing remaining minimum processing workload. All three dispatching
baselines use the same ECT machine-selection rule and global
earliest-gap insertion. Machine ties are resolved by completion time, start
time, processing time, and machine index, in that order.

\subsection{Job-indexed adaptive large-neighborhood search}

J-ALNS represents a solution by an operation sequence over explicitly indexed
job copies. It is initialized independently from the FIFO, SPT, and MWKR
sequences; every candidate is decoded with the same ECT machine selection and
earliest-gap insertion used by the dispatching baselines. The six neighborhoods
relocate a block, reverse a block, shuffle a block, exchange two blocks,
remove and randomly reinsert operations, or pull a job finishing within 2\%
of the incumbent makespan toward the front of the sequence.

Neighborhoods start with equal weights. Rewards of 8, 4, and 1 are assigned to
a new global best, a strict current-solution improvement, and an accepted
nonimproving move, respectively. Every 25 evaluations, each used weight is
updated as $0.75$ times its previous value plus $0.25$ times its mean reward,
with a floor of $0.1$. Simulated-annealing acceptance cools geometrically from
$\max\{1,0.02C_0\}$ to $\max\{0.05,0.0002C_0\}$ over the time limit. After 80
evaluations without a new global best, ruin-and-recreate restarts from the best sequence.
The broad benchmark assigns 10 seconds and one worker to each case. In the
dense fixed-composition study, the budget is the larger of 10 seconds and the measured
cold-start TCTR time. J-ALNS receives no fluid allocation, TCTR template,
stream, or incumbent.

\subsection{Hybrid genetic algorithm}

HGA uses two job-indexed chromosomes: an operation sequence and one eligible
machine-choice gene per operation. The population has 12 individuals
and retains two elites. Three initial individuals use FIFO, SPT, and MWKR
sequences with ECT machine choices; the remaining individuals combine shuffled
or mutated sequences with shortest-processing-time or random machine genes.
Tournament selection uses three contenders. Sequence crossover is a
precedence-preserving order crossover based on a random subset of job copies,
and machine genes use uniform crossover. The crossover probability is
$0.90$; sequence and machine mutation probabilities are $0.35$ and $0.60$.
After each generation, local improvement moves one critical job earlier or
changes one of its flexible machine genes. Every chromosome is decoded by
active earliest-gap insertion.

Each expanded case is run for 10 seconds under three deterministic seeds, one
worker per seed. The case result is the middle run after sorting by makespan and
then by replication index. HGA receives neither fluid information nor TCTR
information.

\subsection{Independent finite-instance CP-SAT benchmark}

The finite-instance CP-SAT model is indexed by job copies. Every
operation has common start and end variables and one optional fixed-duration
interval for each eligible machine. Exactly one interval is present,
successive operations satisfy job precedence, and each machine has a
\texttt{NoOverlap} constraint. A maximum-of-job-ends variable is minimized.
The model starts without an incumbent, objective cap, or solution hint. Each
expanded case receives 300 seconds and one solver worker; unavailable
incumbents remain missing rather than being replaced by another method's
schedule.

\subsection{Software and hardware}

TCTR and the dispatching runs use Python 3.11.5 and OR-Tools 9.14.6206.
The solver-free J-ALNS and HGA archives record Python 3.9.12 and Python 3.10.5,
respectively. The 300-second CP-SAT rerun uses Python 3.12.13 and OR-Tools
9.15.6755. All runs use the same Intel Core i7-9700F workstation with eight
logical cores and 31.9 GB RAM. For the broad benchmark, up to eight cases are
executed concurrently, while every within-case solver is restricted to one
worker.

\section{Controlled recurrence versus finite-horizon comparison}

The controlled case in Section~6.2 of the main paper has one job type,
$N=160$, and processing-time matrix
\[
    \mathbf P=
    \begin{pmatrix}
        10&4\\
        5&3\\
        9&7
    \end{pmatrix}.
\]
The optimal fluid allocation sends the first two positions entirely to
machine 2 and splits the third position as $7/8$ on machine 1 and $1/8$ on
machine 2. Both machine loads are $63/8$, so denominator clearing gives
$(y,T)=(8,63)$ and integer position counts $(0,8)$, $(0,8)$, and $(7,1)$.

At the selected all-idle recurrent boundary, six jobs wait for their third
operation. For $N=160$, the finite construction therefore uses two serial
remainder jobs, fills this six-job boundary inventory, executes 19 recurrent
periods, and drains the boundary inventory. The serial remainder costs 28,
the filling and draining pieces cost 84 in total, and the recurrent core costs
$19\times63=1197$, giving
$C_{\max}^{\mathrm{slot}}=1309$.

The single-type TCTR enumeration considers each integral candidate batch.
It selects $z=8$ (equivalently $\hat q=20$), the integer allocation above, and
a 24-action copy-stage stream at zero rotation. Direct replay gives
$C_{\max}^{\mathrm{TCTR}}=1267$.

Finite optimality is certified at 1263. Machine-load enumeration rules out
horizons below 1263 after accounting for unavoidable startup idle time. At
1263, machine 1 processes one first-position and 139 third-position operations,
while machine 2 processes the remaining counts $(159,160,21)$ without idle
time. A dynamic program whose state records these three cumulative machine-2
counts constructs a precedence-feasible stage word. Job identities are then
linked by FIFO readiness. A separate CP-SAT stage-word model fixes the
certified horizon and allocation, uses the DP word as a search hint, and
minimizes positional agreement subject to producing a different feasible word.
It constructs a second schedule with $C_{\max}=1263$; the load/startup argument,
rather than this hinted run, supplies the independent lower-bound certificate.

For the projected-state comparison, every schedule is replayed through the
event-driven representation and each post-transition state is counted. The displayed
union contains 385 vertices and 409 directed transitions. The slot-linked,
TCTR, DP, and CP-SAT schedules visit 78, 52, 61, and 229 distinct vertices,
respectively. A dominant recurrent support is extracted from the states whose
visit counts lie within one of the maximum panel count and is verified to form
a simple closed support before plotting. These records generate Figure~5 of
the main paper.

\section{Computational results and data}

\subsection{Data inventory}

The accompanying \path{supplementary_data} directory contains the exact
records used in the computational section. The four primary files in
Table~\ref{tab:data_inventory_supp} preserve every reported schedule or stage;
the remaining CSV files provide the derived profile, family-profile, ray,
paired-comparison, and solver-status summaries.

\begin{table}[htbp]
    \centering
    \caption{Primary machine-readable result files.}
    \label{tab:data_inventory_supp}
    \footnotesize
    \setlength{\tabcolsep}{4pt}
    \begin{tabular*}{\textwidth}{@{\extracolsep{\fill}}p{0.43\textwidth}rp{0.39\textwidth}@{}}
        \toprule
        File & Rows & Coverage \\
        \midrule
        \shortstack[l]{\texttt{public\_dense\_190\_}\\\texttt{case\_results.csv}}
        & 22,800 & 24 rays $\times$ 190 targets $\times$ 5 methods \\
        \shortstack[l]{\texttt{public\_multiplicity\_676\_}\\\texttt{case\_results.csv}}
        & 4,732 & 676 cases $\times$ 7 methods \\
        \shortstack[l]{\texttt{public\_multiplicity\_hga\_}\\\texttt{seed\_results.csv}}
        & 2,028 & 676 cases $\times$ 3 HGA seeds \\
        \shortstack[l]{\texttt{tctr\_component\_build\_up\_}\\\texttt{676\_case\_results.csv}}
        & 3,380 & 676 cases $\times$ 5 nested stages \\
        \bottomrule
    \end{tabular*}
\end{table}

The HGA seed file retains every stochastic replication. The HGA row in the
4,732-row case file is the casewise median makespan and also records the
selected replication, the best and worst makespans, and the across-seed
standard deviation.

The accompanying README lists the six dense-study summaries, four
expanded-benchmark summaries, and the five-row component summary, together
with their row counts and contents.

\subsection{Detailed fixed-composition scaling results}

The main paper reports family-level trends; here we retain all 24 individual
trajectories and crossover thresholds. The dense scaling study contains 12
public base instances, two fixed-composition profiles,
and 190 target operation counts per ray. The target grid uses increments of 100
from $10^2$ through $10^4$ and increments of 1,000 through $10^5$. The equal
profile has $\mathbf b=\mathbf1$ and $\mathbf r=\mathbf0$. In the
partial-growth profile, the longer half of the job types has $b_i=1,r_i=0$,
and the shorter half has $b_i=0,r_i=1$, where length is ranked by minimum
processing workload. Thus each method contributes 4,560 schedules, and the five
methods contribute 22,800 schedules in total.

For every realized vector $\boldsymbol\nu(q)=q\mathbf b+\mathbf r$, the fluid
LP is solved for the full vector and
$\operatorname{Gap}_F=(C_{\max}-\rho_F(\boldsymbol\nu(q)))/
\rho_F(\boldsymbol\nu(q))\times100\%$. Figure~\ref{fig:dense_rays_supp}
shows all 24 rays at every target. The
four columns group Barnes, Brandimarte, Dauzere, and Hurink E/R/V; each column
contains three base instances. The symmetric logarithmic vertical scale preserves
zero gaps while separating small positive values.

\begin{figure}[p]
    \centering
    \includegraphics[width=0.87\textwidth,trim=0bp 0bp 0bp 1bp,clip]{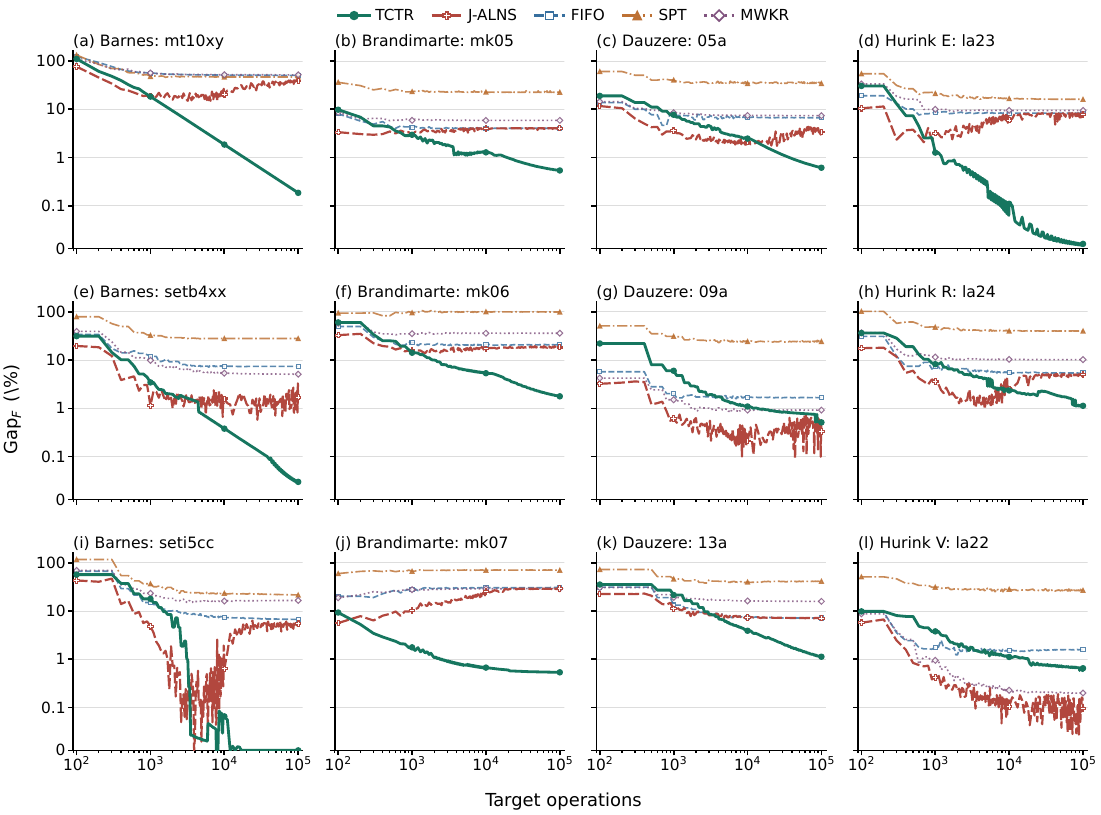}
    \vspace{-2mm}

    \includegraphics[width=0.87\textwidth,trim=0bp 0bp 0bp 1bp,clip]{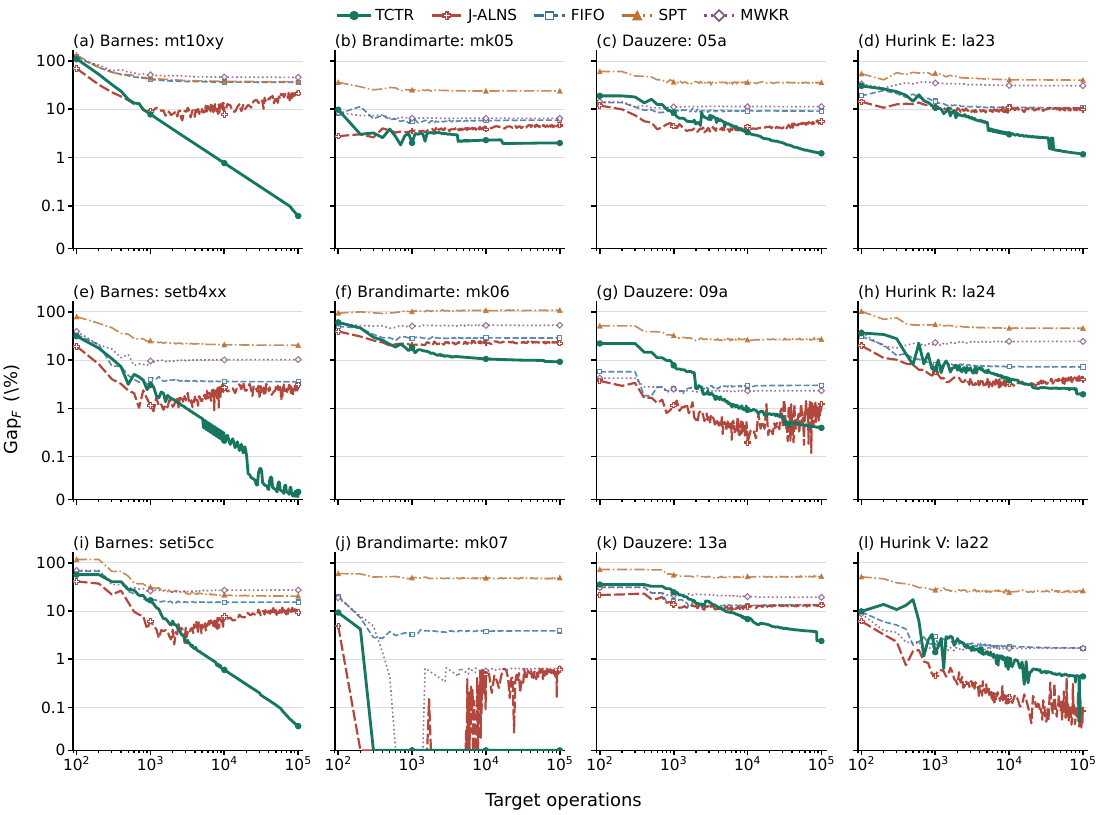}
    \setlength{\abovecaptionskip}{4pt}
    \caption{Per-base-instance fluid-gap trajectories for the equal (top) and
    partial-growth (bottom) profiles.}
    \label{fig:dense_rays_supp}
\end{figure}

\clearpage
\begin{table}[htbp]
    \centering
    \caption{Per-ray TCTR gaps (\%) and the first target from which its gap
    remains lowest among the five methods.}
    \label{tab:dense_ray_summary_supp}
    \scriptsize
    \setlength{\tabcolsep}{3.5pt}
    \begin{tabular*}{\textwidth}{@{\extracolsep{\fill}}llrrrrr@{}}
        \toprule
        Profile & Base instance & $10^2$ & $10^3$ & $10^4$ & $10^5$ & Lowest from \\
        \midrule
        Equal & Barnes: mt10xy & 110.40 & 18.27 & 1.85 & 0.19 & 1{,}000 \\
Equal & Barnes: setb4xx & 31.32 & 3.44 & 0.38 & 0.04 & 4{,}300 \\
Equal & Barnes: seti5cc & 56.64 & 17.82 & 0.08 & 0.00 & 8{,}600 \\
Equal & Brandimarte: mk05 & 9.77 & 2.96 & 1.28 & 0.54 & 700 \\
Equal & Brandimarte: mk06 & 60.31 & 14.21 & 5.34 & 1.77 & 1{,}000 \\
Equal & Brandimarte: mk07 & 9.27 & 1.76 & 0.67 & 0.54 & 200 \\
Equal & Dauzere: 05a & 18.88 & 7.66 & 2.46 & 0.61 & 14{,}000 \\
Equal & Dauzere: 09a & 22.10 & 6.00 & 1.09 & 0.51 & -- \\
Equal & Dauzere: 13a & 35.28 & 21.42 & 3.89 & 1.13 & 3{,}700 \\
Equal & Hurink E: la23 & 30.35 & 1.26 & 0.11 & 0.01 & 1{,}000 \\
Equal & Hurink R: la24 & 36.40 & 8.23 & 2.35 & 1.13 & 10{,}000 \\
Equal & Hurink V: la22 & 9.81 & 3.78 & 1.11 & 0.65 & -- \\
\addlinespace[3pt]
Partial & Barnes: mt10xy & 110.40 & 7.81 & 0.76 & 0.08 & 900 \\
Partial & Barnes: setb4xx & 31.32 & 3.05 & 0.21 & 0.02 & 2{,}500 \\
Partial & Barnes: seti5cc & 56.64 & 16.64 & 0.60 & 0.06 & 3{,}000 \\
Partial & Brandimarte: mk05 & 9.77 & 2.00 & 2.29 & 2.00 & 600 \\
Partial & Brandimarte: mk06 & 60.31 & 18.12 & 10.54 & 9.20 & 600 \\
Partial & Brandimarte: mk07 & 9.27 & 0.00 & 0.00 & 0.00 & 300 \\
Partial & Dauzere: 05a & 18.88 & 8.61 & 3.32 & 1.22 & 7{,}700 \\
Partial & Dauzere: 09a & 22.10 & 7.79 & 0.93 & 0.39 & 74{,}000 \\
Partial & Dauzere: 13a & 35.28 & 24.32 & 6.80 & 2.39 & 4{,}600 \\
Partial & Hurink E: la23 & 30.35 & 10.80 & 3.01 & 1.17 & 1{,}600 \\
Partial & Hurink R: la24 & 36.40 & 6.50 & 3.78 & 1.95 & 18{,}000 \\
Partial & Hurink V: la22 & 9.81 & 1.40 & 0.97 & 0.44 & -- \\
\bottomrule

\end{tabular*}
\end{table}

\subsection{Contribution of TCTR components}

The main paper reports the headline marginal effects; here the component file
and Figure~\ref{fig:tctr_component_build_up_supp} provide their family-resolved
distributions and cumulative runtimes. The file reports core replay, gap
insertion, stream diversification,
adaptive batches, and full TCTR for each of the same 676 cases. Its 3,380 rows
retain the selected batch, order, rotation, machine-selection result, makespan, fluid gap,
incremental reduction, and cold-start runtime at every stage. The accompanying
five-row summary reproduces the aggregates in the main paper.

The median cumulative cold-start times across the five stages are 0.31, 0.50,
5.75, 9.31, and 9.38 seconds. Figure~\ref{fig:tctr_component_build_up_supp}
shows the family-level gap trajectories and the distributions of positive
casewise reductions that underlie the coverage and conditional-median entries
in Table~2 of the main paper.

\begin{figure}[!htbp]
    \centering
    \includegraphics[width=0.82\textwidth,trim=3bp 4bp 4bp 4bp,clip]{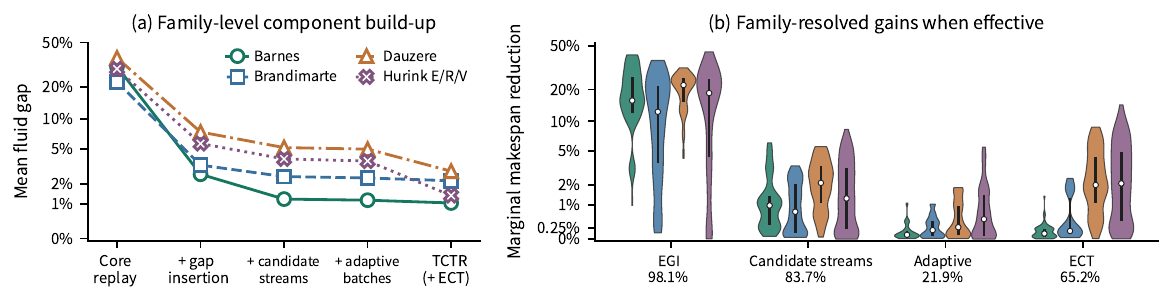}
    \setlength{\abovecaptionskip}{4pt}
    \caption{Nested TCTR component build-up. (a) Family mean fluid gaps.
    (b) Casewise makespan reductions conditional on strict improvement;
    points and bars show medians and interquartile ranges, and labels give
    overall improvement coverage. Both vertical axes use log-spaced tick
    positions.}
    \label{fig:tctr_component_build_up_supp}
\end{figure}

\clearpage
\begin{landscape}
\subsection{Broad finite-horizon benchmark}

The 676 expanded cases cross 169 public base instances with four multiplicity
profiles. Table~\ref{tab:expanded_case_details_supp} gives the case-level
results underlying Table~1 of the main paper. The corresponding 4,732-row CSV
also records exact multiplicity vectors, runtimes, TCTR replay metadata,
J-ALNS search statistics and seeds, HGA case medians and search statistics,
and CP-SAT bounds and statuses. The separate 2,028-row HGA file retains all
three seed-level runs per case.

\singlespacing
\scriptsize
\setlength{\tabcolsep}{2.6pt}
\begin{longtable}{lll*{11}{r}l}
\caption{Case-level results underlying Table~1 of the main paper. Bold marks
the smallest available makespan.}
\label{tab:expanded_case_details_supp}\\
\toprule
Family & Profile & Instance & Ops. & Fluid LB & TCTR & J-ALNS & HGA & FIFO & SPT &
MWKR & CP inc. & CP bound & CP gap (\%) & Status \\
\midrule
\endfirsthead
\multicolumn{15}{l}{\tablename\ \thetable\ continued}\\[2pt]
\toprule
Family & Profile & Instance & Ops. & Fluid LB & TCTR & J-ALNS & HGA & FIFO & SPT &
MWKR & CP inc. & CP bound & CP gap (\%) & Status \\
\midrule
\endhead
\midrule
\multicolumn{15}{r}{Continued on next page} \\
\endfoot
\bottomrule
\endlastfoot
Barnes & Equal & mt10c1 & 10000 & 63100.0 & \textbf{63652} & 69733 & 71379 & 90934 & 84974 & 88190 & 90296 & 63183.0 & 30.0 & FEASIBLE \\
Barnes & Equal & mt10cc & 10000 & 63100.0 & \textbf{63605} & 72318 & 71964 & 84313 & 85027 & 88190 & 90296 & 63183.0 & 30.0 & FEASIBLE \\
Barnes & Equal & mt10x & 10000 & 55600.0 & \textbf{56986} & 70478 & 73808 & 93557 & 81396 & 91506 & 90296 & 55840.0 & 38.2 & FEASIBLE \\
Barnes & Equal & mt10xx & 10000 & 55600.0 & \textbf{56832} & 70071 & 68765 & 94104 & 81396 & 90786 & 90192 & 39798.0 & 55.9 & FEASIBLE \\
Barnes & Equal & mt10xxx & 10000 & 55600.0 & \textbf{56832} & 69232 & 68555 & 94304 & 81396 & 90696 & 90192 & 39798.0 & 55.9 & FEASIBLE \\
Barnes & Equal & mt10xy & 10000 & 54800.0 & \textbf{55814} & 66201 & 68170 & 82675 & 80327 & 83192 & 90296 & 54994.0 & 39.1 & FEASIBLE \\
Barnes & Equal & mt10xyz & 10000 & 53400.0 & \textbf{54305} & 63541 & 59882 & 71447 & 72046 & 76044 & 90296 & 53529.0 & 40.7 & FEASIBLE \\
Barnes & Equal & setb4c9 & 10050 & 57419.0 & \textbf{57472} & 58238 & 57680 & 61276 & 73362 & 62683 & 60388 & 57419.0 & 4.9 & FEASIBLE \\
\end{longtable}
\end{landscape}
\onehalfspacing